\documentclass[reqno]{amsart}

\usepackage{graphicx}
\usepackage{amscd}
\usepackage{amsmath, amssymb}
\usepackage{graphics}
\usepackage{graphicx}
\usepackage{epsfig}
\usepackage{xcolor}
\usepackage{soul}
\usepackage{tikz}

\newtheorem{theorem}{Theorem}

\newtheorem{definition}[theorem]{Definition}

\newtheorem{proposition}[theorem]{Proposition}

\newtheorem{question}[theorem]{Question}

\begin{document}

\title[Local Representations of the Singular Braid Monoid]{Classification of Homogeneous Local Representations of the Singular Braid Monoid}

\author{Taher I. Mayassi}

\address{Taher I. Mayassi\\
Department of Mathematics and Physics\\
Lebanese International University\\
P.O. Box 146404, Beirut, Lebanon}

\email{taher.mayasi@liu.edu.lb}

\author{Mohamad N. Nasser}

\address{Mohamad N. Nasser\\
         Department of Mathematics and Computer Science\\
         Beirut Arab University\\
         P.O. Box 11-5020, Beirut, Lebanon}
         
\email{m.nasser@bau.edu.lb}

\begin{abstract}
For a natural number $n$, denote by $B_n$ the braid group on $n$ strings and by $SM_n$ the singular braid monoid on $n$ strings. $SM_n$ is one of the most important extensions of $B_n$. In \cite{17}, Y. Mikhalchishina classified all homogeneous $2$-local representations of $B_n$ for all $n \geq 3$. In this article, we extend the result of Mikhalchishina in two ways. First, we classify all homogeneous $3$-local representations of $B_n$ for all $n \geq 4$. Second, we classify all homogeneous $2$-local representations of $SM_n$ for all $n\geq 2$ and all homogeneous $3$-local representations of $SM_n$ for all $n\geq 4$.
\end{abstract}

\maketitle

\renewcommand{\thefootnote}{}

\section{Introduction} 

\vspace*{0.1cm}

The braid group on $n$ strings, $B_n$, is an abstract group with the so called Artin generators $\sigma_1,\sigma_2, \ldots,\sigma_{n-1}$. The singular braid monoid, $SM_n$, is the monoid generated by the Artin generators $\sigma_1^{\pm 1},\sigma_2^{\pm 1}, \ldots,\sigma_{n-1}^{\pm 1}$ of $B_n$ and the singular generators $\tau_1,\tau_2, \ldots, \tau_{n-1}$. We show the geometrical shapes of these generators in Section 2. $SM_n$ is one of the most important extensions of $B_n$. Its definition is introduced first by J. Baez in \cite{6} and J. Birman in \cite{7}. In \cite{14}, it is shown that $SM_n$ embeds into a group, $SB_n$, called the singular braid group. For more information on $SM_n$ and $SB_n$, the references \cite{8}, \cite{9}, and \cite{10} are left for the readers.

\vspace*{0.1cm}

Representations of the braid group $B_n$ and its extensions are important to researchers working on Representation Theory. One of the most important representations of $B_n$ are the local representations. A representation $\rho: B_n \rightarrow GL_{m}(\mathbb{Z}[t^{\pm 1}])$ is said to be $k$-local if it is of the form
$$\rho(\sigma_i) =\left( \begin{array}{c|@{}c|c@{}}
   \begin{matrix}
     I_{i-1} 
   \end{matrix} 
      & 0 & 0 \\
      \hline
    0 &\hspace{0.2cm} \begin{matrix}
   		M_i
   		\end{matrix}  & 0  \\
\hline
0 & 0 & I_{n-i-1}
\end{array} \right) \hspace*{0.2cm} \text{for} \hspace*{0.2cm} 1\leq i\leq n-1,$$ 
where $M_i \in M_k(\mathbb{Z}[t^{\pm 1}])$ with $k=m-n+2$ and $I_r$ is the $r\times r$ identity matrix. The $k$-local representation is said to be homogeneous if all the matrices $M_i$ are equal. The known Burau representation of $B_n$ defined in \cite{1} is one of the famous linear representations of $B_n$ which is homogeneous $2$-local.

\vspace*{0.1cm}

In \cite{13}, Bardakov, Chbili, and Kozlovskaya found a type of extensions of representations of $B_n$ to $SM_n$; we call them $\Phi$-type extensions. More precisely, they proved that if $\rho:B_n\rightarrow G_n$ is a representation of $B_n$ to a group $G_n$ and $\mathbb{K}$ is a field, then the map $\Phi_{a,b,c}:SM_n\rightarrow \mathbb{K}[G_n]$ defined by
\begin{align*}
&\Phi_{a,b,c}(\sigma_i^{\pm 1})=\rho(\sigma_i^{\pm 1}) \hspace{0.2cm} \text{and} \hspace{0.2cm} \Phi_{a,b,c}(\tau_i)=a\rho(\sigma_i)+b\rho(\sigma_i^{-1})+ce,\hspace{0.2cm} i=1,2,\ldots,n-1,
\end{align*}
where $a,b,c \in \mathbb{K}$ and $e$ is the neutral element of $G_n$, is a representation of $SM_n$, which is an extension of $\rho$. One of these family of representations is the Birman representation in case $\Phi$ is the identity map, $a=1,b=-1$ and $c=0$ \cite{7}, which is shown to be faithful by L. Paris in \cite{12}. M. Nasser studied the faithfulness of $\Phi_{a,b,c}$ is some cases \cite{11}.

\vspace*{0.1cm}

In \cite{40}, M. Nasser introduced a new definition of $k$-local representations of $SM_n$ that are extensions of $k$-local representations of $B_n$. This type of extensions is called $k$-local extensions. He considered several homogeneous $k$-local representations of $B_n$ and studied the relation between their $k$-local extensions and $\Phi$-type extensions to $SM_n$.

\vspace*{0.1cm}

Y. Mikhalchishina classified all $2$-local representations of $B_3$ and all homogeneous $2$-local representations of $B_n$ for all $n \geq 3$ \cite{17}. Moreover, she investigated the connection of these representations with the Burau representation. In \cite{20}, M. Chreif and M. Dally studied the irreducibility of $2$-local representations of $B_n$ in some cases.

\vspace*{0.1cm}

The aim of this paper is to extend the result of Mikhalchishina as follows. First, in Section 3, we classify all non-trivial homogeneous $2$-local representations of $SM_n$ for all $n\geq 2$ (Theorem \ref{n2} and Theorem \ref{n22}). Moreover, we classify all $\Phi$-type extensions of all non-trivial homogeneous $2$-local representations of $B_n$ to $SM_n$ for all $n\geq 3$ (Theorem \ref{88}). Second, in Section 4, we extend the result of Mikhalchishina and the result in Section 3 by classifying all non-trivial homogeneous $3$-local representations of $B_n$ and $SM_n$ for all $n\geq 4$ (Theorem \ref{repBn} and Theorem \ref{repSMn}).

\section{Generalities} 

\vspace*{0.1cm}

The braid group, $B_n$, has generators $\sigma_1,\sigma_2,\ldots,\sigma_{n-1}$ that satisfy the following relations.
\begin{align*}
&\sigma_i\sigma_{i+1}\sigma_i = \sigma_{i+1}\sigma_i\sigma_{i+1} ,\hspace{0.5cm} i=1,2,\ldots,n-2,\\
&\sigma_i\sigma_j = \sigma_j\sigma_i , \hspace{2.1cm} |i-j|\geq 2.
\end{align*}

On the other hand, the singular braid monoid, $SM_n$, is generated by the generators $\sigma_1^{\pm 1},\sigma_2^{\pm 1}, \ldots,\sigma_{n-1}^{\pm 1}$ of $B_n$ in addition to the singular generators $\tau_1,\tau_2, \ldots, \tau_{n-1}$. The generators $\sigma_i, \sigma_i^{-1},$ and $\tau_i$ of $SM_n$ satisfy the following relations.

\begin{align*}
(1) \hspace{1cm} \sigma_i\sigma_{i+1}\sigma_i = \sigma_{i+1}\sigma_i\sigma_{i+1} ,\hspace{.3cm} &\text{for} \hspace{.3cm} i=1,2,\ldots ,n-2,\vspace{0.1cm}\\ 
(2) \hspace{2.65cm} \sigma_i\sigma_j = \sigma_j\sigma_i , \hspace{.3cm} &\text{for} \hspace{.3cm} |i-j|\geq 2, \vspace{0.1cm}\\
(3) \hspace{2.85cm} \tau_i\tau_j = \tau_j\tau_i , \hspace{.3cm} &\text{for} \hspace{.3cm} |i-j|\geq 2, \vspace{0.1cm}\\
(4) \hspace{2.75cm} \tau_i\sigma_j=\sigma_j\tau_i, \hspace{.3cm} &\text{for} \hspace{.3cm} |i-j|\geq 2,  \vspace{0.1cm}\\
(5) \hspace{2.8cm} \tau_i\sigma_i=\sigma_i\tau_i, \hspace{.3cm} &\text{for} \hspace{.3cm} i=1,2,\ldots ,n-1,  \vspace{0.1cm}\\
(6) \hspace{1.1cm} \sigma_i\sigma_{i+1}\tau_i=\tau_{i+1}\sigma_i\sigma_{i+1}, \hspace{.3cm} &\text{for} \hspace{.3cm} i=1,2,\ldots ,n-2,  \vspace{0.1cm}\\
(7) \hspace{1.1cm} \sigma_{i+1}\sigma_{i}\tau_{i+1}=\tau_{i}\sigma_{i+1}\sigma_{i}, \hspace{.3cm} &\text{for} \hspace{.3cm} i=1,2,\ldots ,n-2.\\
\end{align*}

The relations (4), (5), (6), and (7) are called the mixed relations.

\vspace*{0.1cm}

In the following figures we present the geometrical interpretation of the generators $\sigma_i, \sigma_i^{-1},$ and $\tau_i$ of $SM_n$.

\vspace*{0.1cm}

\begin{center}
\begin{tikzpicture}
 \draw[thick] (-1.5,0)--(-1.5,2); 
    \fill (-1,1) circle(1.5pt) (-1.2,1) circle(1.5pt)(-0.8,1)circle(1.5pt);       
    \draw[thick] (-.5,0)--(-.5,2);
    \draw[thick] (2.5,0)--(2.5,2); 
    \fill (2,1) circle(1.5pt) (2.2,1) circle(1.5pt)
    (1.8,1)circle(1.5pt);   
        
    \draw[thick] (1.5,0)--(1.5,2);
    \draw[thick] (0,0) to[out=90, in=-90] (1,2);
    \draw[thick, white, line width=4pt] (1,0) to[out=90, in=-90] (0,2); 
    \draw[thick] (1,0) to[out=90, in=-90] (0,2);
    \node[above] at(0,2){$i$};
    \node[above] at(1,2){$i+1$};
    
\node at (3.5,1){,};    

 \draw[thick] (5.5,0)--(5.5,2); 
    \fill (5,1) circle(1.5pt) (4.8,1) circle(1.5pt)(5.2,1)circle(1.5pt);       
    \draw[thick] (4.5,0)--(4.5,2);
    \draw[thick] (8.5,0)--(8.5,2); 
    \fill (8,1) circle(1.5pt) (8.2,1) circle(1.5pt)(7.8,1)circle(1.5pt);       
    \draw[thick] (7.5,0)--(7.5,2);    
    
\draw[thick] (7,0) to[out=90, in=-90] (6,2);
    \draw[thick, white, line width=4pt] (6,0) to[out=90, in=-90] (7,2); 
    \draw[thick] (6,0) to[out=90, in=-90] (7,2);
    \node[above] at(6,2){$i$};
    \node[above] at (7,2) {$i+1$};
    
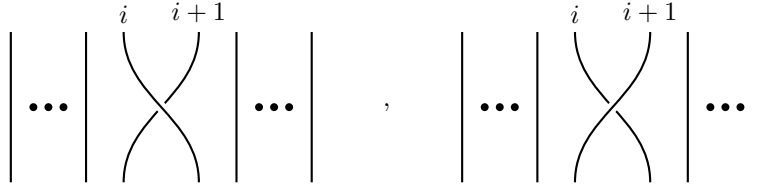
    \node at (3.5, -0.5) {The braid generators \(\sigma_i\) and \(\sigma_i^{-1}\)};
\end{tikzpicture}
\end{center}
\vspace{0.2cm}
\begin{center}
\begin{tikzpicture}
    \draw[thick] (-1.5,0)--(-1.5,2); 
    \fill (-1,1) circle(1.5pt) (-1.2,1) circle(1.5pt)(-0.8,1)circle(1.5pt);       
    \draw[thick] (-.5,0)--(-.5,2);
    \draw[thick] (2.5,0)--(2.5,2); 
    \fill (2,1) circle(1.5pt) (2.2,1) circle(1.5pt)(1.8,1)circle(1.5pt);       
    \draw[thick] (1.5,0)--(1.5,2);
    \draw[thick] (0,0) to[out=90, in=-90] (1,2);
    \draw[thick, white, line width=4pt] (1,0) to[out=90, in=-90] (0,2); 
    \draw[thick] (1,0) to[out=90, in=-90] (0,2);
    \fill[black] (0.5,1) circle (2pt); 
    \node[above] at(0,2){$i$};
    \node[above] at(1,2){$i+1$};
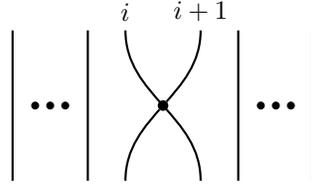
    \node at (0.5, -0.5) {The singular generators \(\tau_i\)};
\end{tikzpicture}
\end{center}

\vspace*{0.1cm}

Now, we give the concept of $k$-local representations of the braid group $B_n$.

\vspace*{0.1cm}

\begin{definition}
A representation $\rho: B_n \rightarrow GL_{m}(\mathbb{Z}[t^{\pm 1}])$ is said to be $k$-local if it is of the form
$$\rho(\sigma_i) =\left( \begin{array}{c|@{}c|c@{}}
   \begin{matrix}
     I_{i-1} 
   \end{matrix} 
      & 0 & 0 \\
      \hline
    0 &\hspace{0.2cm} \begin{matrix}
   		M_i
   		\end{matrix}  & 0  \\
\hline
0 & 0 & I_{n-i-1}
\end{array} \right) \hspace*{0.2cm} \text{for} \hspace*{0.2cm} 1\leq i\leq n-1,$$ 
where $M_i \in M_k(\mathbb{Z}[t^{\pm 1}])$ with $k=m-n+2$ and $I_r$ is the $r\times r$ identity matrix. The $k$-local representation is said to be homogeneous if all the matrices $M_i$ are equal.
\end{definition}

\vspace*{0.1cm}

In the next definition, we extend the concept of $k$-local representations of the braid group $B_n$ to the singular braid monoid $SM_n$.

\vspace*{0.1cm}

\begin{definition} \label{LocalSM}
Let $\rho: B_n \rightarrow GL_{m}(\mathbb{Z}[t^{\pm 1}])$ be a $k$-local representation. A $k$-local extension of $\rho$ to $SM_n$ is a representation $\rho': SM_n \rightarrow M_{m}(\mathbb{Z}[t^{\pm 1}])$ that extends $\rho$ and has the form

$$\rho'(\sigma_i)=\rho(\sigma_i) \hspace*{0.2cm} \text{for} \hspace*{0.2cm} 1\leq i\leq n-1,$$
and
$$\rho'(\tau_i) =\left( \begin{array}{c|@{}c|c@{}}
   \begin{matrix}
     I_{i-1} 
   \end{matrix} 
      & 0 & 0 \\
      \hline
    0 &\hspace{0.2cm} \begin{matrix}
   		N_i
   		\end{matrix}  & 0  \\
\hline
0 & 0 & I_{n-i-1}
\end{array} \right) \hspace*{0.2cm} \text{for} \hspace*{0.2cm} 1\leq i\leq n-1,$$ 
where $N_i \in M_k(\mathbb{Z}[t^{\pm 1}])$ with $k=m-n+2$ and $I_r$ is the $r\times r$ identity matrix. The $k$-local extension is said to be homogeneous if all the matrices $N_i$ are equal.
\end{definition}

\vspace*{0.1cm}

In what follows, we define the concept of $\Phi$-type extensions of representations of the braid group $B_n$ to the singular braid monoid $SM_n$. First, we introduce a proposition that was given by Bardakov, Chbili, and Kozlovskaya in \cite{13}.

\vspace*{0.1cm}

\begin{proposition} \cite{13}
Let $\rho: B_n \rightarrow G_n$ be a representation of the braid group $B_n$ to a group $G_n$ and let $\mathbb{K}$ be a field with $a,b,c \in \mathbb{K}$. Then, the map $\Phi_{a,b,c}:SM_n\rightarrow \mathbb{K}[G_n]$ which acts on the generators of $SM_n$ by the rules
\begin{align*}
&\Phi_{a,b,c}(\sigma_i^{\pm 1})=\rho(\sigma_i^{\pm 1}) \hspace{0.2cm} \text{and} \hspace{0.2cm} \Phi_{a,b,c}(\tau_i)=a\rho(\sigma_i)+b\rho(\sigma_i^{-1})+ce,\hspace{0.2cm} i=1,2,\ldots,n-1,
\end{align*}
defines a representation of $SM_n$ to $\mathbb{K}[G_n]$. Here $e$ is a neutral element of $G_n$.
\end{proposition}

\vspace*{0.1cm}

Now, we give the definition of $\Phi$-type extensions of representations of braid group $B_n$ to the singular braid monoid $SM_n$.

\vspace*{0.1cm}

\begin{definition}
Let $\rho: B_n \rightarrow G_n$ be a representation of the braid group $B_n$ to a group $G_n$ and let $\hat{\rho}: SM_n \rightarrow \mathbb{K}[G_n]$ be a representation of $SM_n$ that extends $\rho$, where $\mathbb{K}$ is a field. We say that $\hat{\rho}$ is a $\Phi$-type extension of $\rho$ to $SM_n$ if there exist $a,b,c \in \mathbb{K}$ such that $\hat{\rho}=\Phi_{a,b,c}$.
\end{definition}

\vspace*{0.1cm}

In what follows, we give the main definitions and basic results of specific two local representations of the braid group $B_n$. The first representation is the known Burau representation, which is a homogeneous $2$-local representation of $B_n$. The second representation is the $F$-representation, which is a homogeneous $3$-local representation of $B_n$. The importance of choosing such representations is to present some special cases of our obtained results.

\vspace*{0.1cm}

We start by the main definitions and results of the homogeneous $2$-local representation of $B_n$, the known Burau representation.

\vspace*{0.1cm}

\begin{definition} \cite{1} \label{defBurau}
Let $t$ be indeterminate. The Burau representation $\rho_B: B_n\rightarrow GL_n(\mathbb{Z}[t^{\pm 1}])$ is the representation given by
$$\sigma_i\rightarrow \left( \begin{array}{c|@{}c|c@{}}
   \begin{matrix}
     I_{i-1} 
   \end{matrix} 
      & 0 & 0 \\
      \hline
    0 &\hspace{0.2cm} \begin{matrix}
   	1-t & t\\
   	1 & 0\\
\end{matrix}  & 0  \\
\hline
0 & 0 & I_{n-i-1}
\end{array} \right) \hspace*{0.2cm} \text{for} \hspace*{0.2cm} 1\leq i\leq n-1.$$ 
\end{definition}

\vspace*{0.1cm}

It has been proven in \cite{75} that the Burau representation is reducible for $n\geq 3$, and the complex specialization of the reduced Burau representation is given in the following.

\vspace*{0.1cm}

\begin{theorem}\cite{75} \label{bure}
The Burau representation is reducible for $n\geq 3$, and the reduced Burau representation $\mu_{B}: B_n\rightarrow GL_{n-1}(\mathbb{Z}[t^{\pm 1}])$ is defined by\\

$\sigma_1\rightarrow \left( \begin{array}{c|@{}c@{}}
   \begin{matrix}
     -t & 1 
   \end{matrix} & 0 \\
      \begin{matrix}
     0 & 1 
   \end{matrix} & 0 \\
   \hline
         \begin{matrix}
     0 & 0 
   \end{matrix} & I_{n-3} \\
\end{array} \right), \hspace{0.1cm} \sigma_{n-1}\rightarrow \left( \begin{array}{c|@{}c@{}}
   I_{n-3} & \begin{matrix}
   0 & 0 \\
   \end{matrix} \\
   \hline
   0 & \begin{matrix}
   \hspace*{0.1cm} 1 & 0 \\
   \end{matrix}\\
   0 & \begin{matrix}
   \hspace*{0.1cm} t & -t \\
   \end{matrix}\\
\end{array} \right),$ and

$$\sigma_i\mapsto \left( \begin{array}{c|@{}c|c@{}}
   \begin{matrix}
     I_{i-2} 
   \end{matrix} 
       &  0 & 0\\
      \hline
    0 &\hspace{0.2cm} \begin{matrix}
   	1 & 0 & 0 \\
   	t & -t & 1 \\
   	0 & 0 & 1 \\
\end{matrix}  & 0  \\
\hline
0 & 0  & I_{n-i-2}
\end{array} \right) \hspace*{0.2cm} for \hspace*{0.2cm} 2\leq i \leq n-2.$$ 
\end{theorem}

\begin{theorem}\cite{75}
For $z \in \mathbb{C}$, the reduced burau representation $\mu_B: B_n \mapsto GL_{n-1}(\mathbb{C})$ is irreducible if and only if $z$ is not a root of the polynomial $f_n(x)=x^{n-1}+x^{n-2}+\ldots+x+1.$
\end{theorem}

\vspace*{0.1cm}

Now we give the main definitions and results of the homogeneous $3$-local representation of $B_n$, the $F$-representation, which is given by V. Bardakov and P. Bellingeri in \cite{19}.

\vspace*{0.1cm}

\begin{definition} \cite{19} \label{Fdef}
Let $t$ be indeterminate and let $\{v_o,v_1,\ldots, v_n\}$ be a basis of the free $\mathbb{Z}[t^{\pm 1}]$-module of rank $n+1$. We define the $F$-representation $\rho_F: B_n \rightarrow GL_{n+1}(\mathbb{Z}[t^{\pm 1}])$ by
\begin{align*}
\rho_F(\sigma_i):
\left\{\begin{array}{l}
v_i\rightarrow v_{i-1}-tv_i+tv_{i+1},\\
v_k\rightarrow v_k \hspace*{3 cm} k \neq i,\\
\end{array} \right. 
\end{align*}
for $1\leq i\leq n-1$. That is,\vspace*{0.2cm}
$$\sigma_i\rightarrow \left( \begin{array}{c|@{}c|c@{}}
   \begin{matrix}
     I_{i-1} 
   \end{matrix} 
      & 0 & 0 \\
      \hline
    0 &\hspace{0.2cm} \begin{matrix}
   		1 & 1 & 0 \\
   		0 &  -t & 0 \\   		
   		0 &  t & 1 \\
   		\end{matrix}  & 0  \\
\hline
0 & 0 & I_{n-i-1}
\end{array} \right) \hspace*{0.2cm} \text{for} \hspace*{0.2cm} 1\leq i\leq n-1.$$ 
\end{definition}

\vspace*{0.1cm}

It has been proven in \cite{55} that the $F$-representation is reducible for $n\geq 3$, and the reduced $F$-representation is given in the following.

\vspace*{0.1cm}

\begin{theorem}\cite{55} \label{Fre}
The $F$-representation is reducible for $n\geq 3$, and the reduced $F$-representation $\mu_{F}: B_n\mapsto GL_{n-1}(\mathbb{Z}[t^{\pm 1}])$ is defined by

$$\sigma_1\rightarrow \left( \begin{array}{c@{}c|c@{}}
   \begin{matrix}
     -1-t & 0 & 0 & 0 & \ldots & 0\\ 
     -1+t & 1 & 0 & 0 & \ldots & 0 \\
     -1 & 0 & 1 & 0 & \ldots & 0 \\

     \vdots & \vdots & & \ddots &  & \vdots \\
     -1 & 0 & \ldots & 0 & 1 & 0 \\
     -1 & 0 & \ldots & 0 & 0 & 1 \\
   \end{matrix} 
    
\end{array} \right), \hspace{0.1cm} \sigma_{n-1}\rightarrow \left( \begin{array}{c|@{}c c@{}}
   \begin{matrix}
     I_{n-3} 
   \end{matrix} 
      & 0  \\
      \hline
    0 &\hspace{0.2cm} \begin{matrix}
   		1 &  1  \\
   		0 &  -t  \\   		
   		\end{matrix}  \\
\end{array} \right), \text{ and}$$

$$\sigma_i\rightarrow \left( \begin{array}{c|@{}c|c@{}}
   \begin{matrix}
     I_{i-2} 
   \end{matrix} 
      & 0 & 0 \\
      \hline
    0 &\hspace{0.2cm} \begin{matrix}
   		1 &  1 & 0 \\
   		0 &  -t & 0 \\   		
   		0 &  t & 1 \\
   		\end{matrix}  & 0  \\
\hline
0 & 0 & I_{n-i-2}
\end{array} \right) \hspace*{0.2cm} for \hspace*{0.2cm} 2\leq i\leq n-2 \hspace*{0.2cm}.$$ 

\end{theorem}

\vspace*{0.1cm}

\begin{theorem} \cite{55}
For $z \in \mathbb{C}$, the representation $\mu_F: B_n \rightarrow GL_{n-1}(\mathbb{C})$ is irreducible if and only if $z$ is not a root of the polynomial $g_n(x)=x^{n-1}+2x^{n-2}+3x^{n-3}+\ldots+(n-1)x+n$.
\end{theorem}

\vspace*{0.1cm}

\section{Homogeneous $2$-Local Extensions of Homogeneous $2$-Local Representations of $B_n$ to $SM_n$} 

\vspace*{0.1cm}

In \cite{17}, Mikhalchishina classified all homogeneous $2$-local representations of the braid group $B_n$ for all $n\geq 3$. In this section, we classify all homogeneous $2$-local extensions of all homogeneous $2$-local representations of the braid group $B_n$ to the singular braid monoid $SM_n$ for all $n\geq 2$. So, we extend the result of Mikhalchishina to $SM_n$ for all $n\geq 2$.

\vspace*{0.1cm}

First of all, we consider the case $n=2$, which is a special case, since we have just one relation between the generators of $SM_2$.

\vspace*{0.1cm}

\begin{theorem} \label{n2}
Set $n=2$ and let $\Psi: B_2 \rightarrow GL_2(\mathbb{C})$ be a non-trivial homogeneous $2$-local representation of $B_2$. Let $\Psi': SM_2 \rightarrow M_2(\mathbb{C})$ be a non-trivial homogeneous $2$-local extension of $\Psi$ to $SM_2$. Then, $\Psi'$ is equivalent to one of the following seven representations.
\begin{itemize}
\item[(1)]$\Psi'_1: SM_2 \rightarrow M_2(\mathbb{C})$ such that $\Psi'_1(\sigma_1) =\left( \begin{array}{c@{}}
   \begin{matrix}
   		a & 0\\
   		0 & d
   		\end{matrix}
\end{array} \right)\hspace*{0.15cm} \text{and} \hspace*{0.15cm} \Psi'_1(\tau_1) =\left( \begin{array}{c@{}}
   \begin{matrix}
   		x & 0\\
   		0 & t
   		\end{matrix}
\end{array} \right),\\
\text{where} \hspace*{0.15cm} a, d, x, t \in \mathbb{C}, a\neq d, ad\neq 0.$
\item[(2)]$\Psi'_2: SM_2 \rightarrow M_2(\mathbb{C})$ such that $\Psi'_2(\sigma_1) =\left( \begin{array}{c@{}}
   \begin{matrix}
   		a & 0\\
   		c & a
   		\end{matrix}
\end{array} \right)\hspace*{0.15cm} \text{and} \hspace*{0.15cm} \Psi'_2(\tau_1) =\left( \begin{array}{c@{}}
   \begin{matrix}
   		x & 0\\
   		z & x
   		\end{matrix}
\end{array} \right),\\
\text{where} \hspace*{0.15cm} a, c, x, z \in \mathbb{C}, a \neq 0, c\neq 0.$
\item[(3)]$\Psi'_3: SM_2 \rightarrow M_2(\mathbb{C})$ such that $\Psi'_3(\sigma_1) =\left( \begin{array}{c@{}}
   \begin{matrix}
   		a & 0\\
   		c & d
   		\end{matrix}
\end{array} \right)\hspace*{0.15cm} \text{and} \hspace*{0.15cm} \Psi'_3(\tau_1) =\left( \begin{array}{c@{}}
   \begin{matrix}
   		x & 0\\
   		z & t
   		\end{matrix}
\end{array} \right),\\
\text{where} \hspace*{0.15cm} a, c, d, x, z, t \in \mathbb{C}, ad \neq 0, c\neq 0, a\neq d, c(t-x)+z(a-d)=0.$
\item[(4)]$\Psi'_4: SM_2 \rightarrow M_2(\mathbb{C})$ such that $\Psi'_4(\sigma_1) =\left( \begin{array}{c@{}}
   \begin{matrix}
   		a & b\\
   		0 & a
   		\end{matrix}
\end{array} \right)\hspace*{0.15cm} \text{and} \hspace*{0.15cm} \Psi'_4(\tau_1) =\left( \begin{array}{c@{}}
   \begin{matrix}
   		x & y\\
   		0 & x
   		\end{matrix}
\end{array} \right),\\
\text{where} \hspace*{0.15cm} a, b, x, y \in \mathbb{C}, a \neq 0, b \neq 0.$
\item[(5)]$\Psi'_5: SM_2 \rightarrow M_2(\mathbb{C})$ such that $\Psi'_5(\sigma_1) =\left( \begin{array}{c@{}}
   \begin{matrix}
   		a & b\\
   		0 & d
   		\end{matrix}
\end{array} \right)\hspace*{0.15cm} \text{and} \hspace*{0.15cm} \Psi'_5(\tau_1) =\left( \begin{array}{c@{}}
   \begin{matrix}
   		x & y\\
   		0 & t
   		\end{matrix}
\end{array} \right),\\
\text{where} \hspace*{0.15cm} a, b, d, x, y, t \in \mathbb{C}, ad \neq 0, b\neq 0, a\neq d, b(t-x)+y(a-d)=0$.
\item[(6)]$\Psi'_6: SM_2 \rightarrow M_2(\mathbb{C})$ such that $\Psi'_6(\sigma_1) =\left( \begin{array}{c@{}}
   \begin{matrix}
   		a & b\\
   		c & a
   		\end{matrix}
\end{array} \right)\hspace*{0.15cm} \text{and} \hspace*{0.15cm} \Psi'_6(\tau_1) =\left( \begin{array}{c@{}}
   \begin{matrix}
   		x & y\\
   		z & x
   		\end{matrix}
\end{array} \right),\\
\text{where} \hspace*{0.15cm} a, b, c, x, y, z \in \mathbb{C}, a^2-bc\neq 0, c\neq 0 , b\neq 0, bz=cy.$
\item[(7)]$\Psi'_7: SM_2 \rightarrow M_2(\mathbb{C})$ such that $\Psi'_7(\sigma_1) =\left( \begin{array}{c@{}}
   \begin{matrix}
   		a & b\\
   		c & d
   		\end{matrix}
\end{array} \right)\hspace*{0.15cm} \text{and} \hspace*{0.15cm} \Psi'_7(\tau_1) =\left( \begin{array}{c@{}}
   \begin{matrix}
   		x & y\\
   		z & t
   		\end{matrix}
\end{array} \right),\\
\text{where} \hspace*{0.15cm} a, b, c, d, x, y, z, t \in \mathbb{C}, ad-bc \neq 0, b\neq 0, c\neq 0, a\neq d, bz=cy, b(t-x)=y(a-d), c(t-x)=z(a-d).$
\end{itemize}
\noindent Moreover, if $\Psi'(\tau_1)$ is invertible, then $\Psi'$ is a representation of $SB_n$.
\end{theorem}
\begin{proof}
Set $\Psi'(\sigma_1) =\left( \begin{array}{c@{}}
   \begin{matrix}
   		a & b\\
   		c & d
   		\end{matrix}
\end{array} \right)\hspace*{0.2cm} \text{and} \hspace*{0.2cm} \Psi'(\tau_1) =\left( \begin{array}{c@{}}
   \begin{matrix}
   		x & y\\
   		z & t
   		\end{matrix}
\end{array} \right), \hspace*{0.15cm} \text{such that} \hspace*{0.15cm} a, b, c, d, x, y, z, t \in \mathbb{C}, ad-bc \neq 0.$ The only relation between the generators of $SM_2$ is $\sigma_1\tau_1=\tau_1\sigma_1$. Applying this relation, we get the following three equations.
\begin{equation} \label{eqn1}
bz-cy=0
\end{equation} 
\begin{equation} \label{eqn2}
b(t-x)+y(a-d)=0
\end{equation} 
\begin{equation} \label{eqn3}
c(t-x)+z(a-d)=0
\end{equation} 
We consider and discuss eight cases in the following.
\begin{itemize}
\item[(0)] $b=0, c=0, a=d$: This case is the trivial case, which is rejected since $\Psi'$ is assumed to be non-trivial.
\vspace*{0.1cm}
\item[(1)] $b=0, c=0, a\neq d$: Using the equations (\ref{eqn2}) and (\ref{eqn3}) we get that $y=z=0$, and so we get that $\Psi'$ is equivalent to $\Psi'_1$, as required.
\vspace*{0.1cm}
\item[(2)] $b=0, c\neq 0, a=d$: Using the equations (\ref{eqn1}) and (\ref{eqn3}) we get that $y=0$ and $t=x$, and so we get that $\Psi'$ is equivalent to $\Psi'_2$, as required.
\vspace*{0.1cm}
\item[(3)] $b=0, c\neq 0, a\neq d$: Using the equations (\ref{eqn1}) and (\ref{eqn3}) we get that $y=0$ and $c(t-x)+z(a-d)=0$, and so get that $\Psi'$ is equivalent to $\Psi'_3$, as required.
\vspace*{0.1cm}
\item[(4)] $b\neq 0, c=0, a=d$: Using the equations (\ref{eqn1}) and (\ref{eqn2}) we get that $z=0$ and $t=x$, and so we get that $\Psi'$ is equivalent to $\Psi'_4$, as required.
\vspace*{0.1cm}
\item[(5)] $b\neq 0, c=0, a\neq d$: Using the equations (\ref{eqn1}) and (\ref{eqn2}) we get that $z=0$ and $b(t-x)+y(a-d)=0$, and so we get that $\Psi'$ is equivalent to $\Psi'_5$, as required.
\vspace*{0.1cm}
\item[(6)] $b\neq 0, c\neq 0, a=d$: Using the equations (\ref{eqn1}) and (\ref{eqn2}), we get that $bz=cy$ and $t=x$, and so we get that $\Psi'$ is equivalent to $\Psi'_6$, as required.
\vspace*{0.1cm}
\item[(7)] $b\neq 0, c\neq 0, a\neq d$: Using the equations (\ref{eqn1}), (\ref{eqn2}) and (\ref{eqn3}) we get that $\Psi'$ is equivalent to $\Psi'_7$, as required.
\end{itemize}
\end{proof}

Now, we consider the case $n\geq 3$. In the following theorem, we set a previous result done by Mikhalchishina.

\begin{theorem} \cite{17} \label{ThmM}
Consider $n\geq 3$ and let $\rho: B_n \rightarrow GL_n(\mathbb{C})$ be a non-trivial homogeneous $2$-local representation of $B_n$. Then, $\rho$ is equivalent to one of the following three representations.
\begin{itemize}
\item[(1)] $\rho_1: B_n \rightarrow GL_n(\mathbb{C}) \hspace*{0.15cm} \text{such that } 
\rho_1(\sigma_i) =\left( \begin{array}{c|@{}c|c@{}}
   \begin{matrix}
     I_{i-1} 
   \end{matrix} 
      & 0 & 0 \\
      \hline
    0 &\hspace{0.2cm} \begin{matrix}
   		a & \frac{1-a}{c}\\
   		c & 0
   		\end{matrix}  & 0  \\
\hline
0 & 0 & I_{n-i-1}
\end{array} \right), \\ \text{ where } c \neq 0, a\neq 1, \hspace*{0.15cm} \text{for} \hspace*{0.2cm} 1\leq i\leq n-1.$
\item[(2)] $\rho_2: B_n \rightarrow GL_n(\mathbb{C}) \hspace*{0.15cm} \text{such that }
\rho_2(\sigma_i) =\left( \begin{array}{c|@{}c|c@{}}
   \begin{matrix}
     I_{i-1} 
   \end{matrix} 
      & 0 & 0 \\
      \hline
    0 &\hspace{0.2cm} \begin{matrix}
   		0 & \frac{1-d}{c}\\
   		c & d
   		\end{matrix}  & 0  \\
\hline
0 & 0 & I_{n-i-1}
\end{array} \right),\\ \text{ where } c\neq 0, d\neq 1 \hspace*{0.2cm} \text{for} \hspace*{0.2cm} 1\leq i\leq n-1.$
\item[(3)] $\rho_3: B_n \rightarrow GL_n(\mathbb{C}) \hspace*{0.15cm} \text{such that } 
\rho_3(\sigma_i) =\left( \begin{array}{c|@{}c|c@{}}
   \begin{matrix}
     I_{i-1} 
   \end{matrix} 
      & 0 & 0 \\
      \hline
    0 &\hspace{0.2cm} \begin{matrix}
   		0 & b\\
   		c & 0
   		\end{matrix}  & 0  \\
\hline
0 & 0 & I_{n-i-1}
\end{array} \right),\\ \text{ where } bc\neq 0 \hspace*{0.2cm} \text{for} \hspace*{0.2cm} 1\leq i\leq n-1.$
\end{itemize}
\end{theorem}

\vspace*{0.1cm}

We can easily see that for all $n\geq 3$, the complex specialization of the Burau representation defined in Definition \ref{defBurau} is a special representation of the first family of representation in Theorem \ref{ThmM}. That is, $\rho_B=\rho_1$ in the case $a=1-t$ and $c=1$.

\vspace*{0.1cm}

In the following theorem, we extend the result of Theorem \ref{ThmM} on $B_n$ to the singular braid monoid $SM_n$.

\vspace*{0.1cm}

\begin{theorem} \label{n22}
Consider $n\geq 3$ and let $\rho: B_n \rightarrow GL_n(\mathbb{C})$ be a non-trivial homogeneous $2$-local representation of $B_n$. Let $\rho': SM_n \rightarrow M_n(\mathbb{C})$ be a non-trivial homogeneous $2$-local extension of $\rho$ to $SM_n$. Then, $\rho'$ is equivalent to one of the following three representations. 
\begin{itemize}
\item[(1)] $\rho'_1: SM_n \rightarrow M_n(\mathbb{C}) \hspace*{0.15cm} \text{such that}\\ \hspace*{0.15cm} \rho'_1(\sigma_i) =\rho_1(\sigma_i) \hspace*{0.15cm} \text{and} \hspace*{0.15cm} \rho'_1(\tau_i) =\left( \begin{array}{c|@{}c|c@{}}
   \begin{matrix}
     I_{i-1} 
   \end{matrix} 
      & 0 & 0 \\
      \hline
    0 &\hspace{0.2cm} \begin{matrix}
   		 1-(1-a)(1-t) & \frac{1-a}{c}(1-t)\\
   		c(1-t) & t
   		\end{matrix}  & 0  \\
\hline
0 & 0 & I_{n-i-1}
\end{array} \right), \\
\text{where} \hspace*{0.15cm} a, c, t \in \mathbb{C},a\neq 0, c\neq 0, \hspace*{0.15cm} \text{for} \hspace*{0.15cm} 1\leq i\leq n-1.$
\item[(2)] $\rho'_2:SM_n \rightarrow M_n(\mathbb{C}) \hspace*{0.15cm} \text{such that}\\ \hspace*{0.15cm} \rho'_2(\sigma_i)=\rho_2(\sigma_i) \hspace*{0.15cm} \text{and} \hspace*{0.15cm} \rho'_2(\tau_i) =\left( \begin{array}{c|@{}c|c@{}}
   \begin{matrix}
     I_{i-1} 
   \end{matrix} 
      & 0 & 0 \\
      \hline
    0 &\hspace{0.2cm} \begin{matrix}
   		x & \frac{1-d}{c}(1-x)\\
   		c(1-x) & 1-(1-d)(1-x)
   		\end{matrix}  & 0  \\   		
\hline
0 & 0 & I_{n-i-1}
\end{array} \right), \\
\text{where} \hspace*{0.15cm} c, d, x \in \mathbb{C}, c\neq 0, d\neq 0, \hspace*{0.15cm} \text{for} \hspace*{0.15cm} 1\leq i\leq n-1.$ 
\item[(3)] $\rho'_3:SM_n \rightarrow M_n(\mathbb{C}) \hspace*{0.15cm} \text{such that}\\ \hspace*{0.15cm} \rho'_3(\sigma_i)=\rho_3(\sigma_i) \hspace*{0.15cm}  \text{and} \hspace*{0.15cm} \rho'_3(\tau_i) =\left( \begin{array}{c|@{}c|c@{}}
   \begin{matrix}
     I_{i-1} 
   \end{matrix} 
      & 0 & 0 \\
      \hline
    0 &\hspace{0.2cm} \begin{matrix}
   		x & y\\
   		\frac{cy}{b} & x
   		\end{matrix}  & 0  \\
\hline
0 & 0 & I_{n-i-1}
\end{array} \right), \\
\text{where} \hspace*{0.15cm} b, c, x, y \in \mathbb{C}, b\neq 0, c\neq 0, \hspace*{0.15cm} \text{for} \hspace*{0.15cm} 1\leq i\leq n-1.$ 
\end{itemize}
\noindent Moreover, if $\rho'(\tau_i)$ is invertible for all $1\leq i\leq n-1$, then $\rho'$ is a representation of $SB_n$.
\end{theorem}
\begin{proof}
Set $\rho'(\sigma_i) =\left( \begin{array}{c|@{}c|c@{}}
   \begin{matrix}
     I_{i-1} 
   \end{matrix} 
      & 0 & 0 \\
      \hline
    0 &\hspace{0.2cm} \begin{matrix}
   		a & b\\
   		c & d
   		\end{matrix}  & 0  \\
\hline
0 & 0 & I_{n-i-1}
\end{array} \right) \hspace*{0.15cm} \text{and} \hspace*{0.15cm} \rho'(\tau_i) =\left( \begin{array}{c|@{}c|c@{}}
   \begin{matrix}
     I_{i-1} 
   \end{matrix} 
      & 0 & 0 \\
      \hline
    0 &\hspace{0.2cm} \begin{matrix}
   		x & y\\
   		z & t
   		\end{matrix}  & 0  \\
\hline
0 & 0 & I_{n-i-1}
\end{array} \right), \\
\text{where} \hspace*{0.15cm} a, b, c, d, x, y, z, t \in \mathbb{C}, ad-bc \neq 0, \hspace*{0.15cm} \text{for} \hspace*{0.15cm} 1\leq i\leq n-1.$

\vspace*{0.1cm}

Direct computations show that enough to use three mixed relations, and the other relations implies similar equations. The three mixed relations are: $\tau_1\sigma_1=\sigma_1\tau_1$, $\sigma_1\sigma_2\tau_1=\tau_2\sigma_1\sigma_2$ and $\sigma_2\sigma_1\tau_2=\tau_1\sigma_2\sigma_1$. We consider each relation and write the obtained equations in the following.
\begin{itemize}
\item[(i)] Using the relation $\tau_1\sigma_1=\sigma_1\tau_1$, we obtain the following equations.
\begin{equation} \label{eq4}
bz-cy=0
\end{equation} 
\begin{equation} \label{eq5}
b(t-x)+y(a-d)=0
\end{equation} 
\begin{equation} \label{eq6}
c(t-x)+z(a-d)=0
\end{equation} 
\item[(ii)] Using the relation $\sigma_1\sigma_2\tau_1=\tau_2\sigma_1\sigma_2$, we obtain the following equations.
\begin{equation} \label{eq7}
a(-1+x+bz)=0
\end{equation} 
\begin{equation} \label{eq8}
a(b(t-1)+y)=0
\end{equation} 
\begin{equation} \label{eq9}
adz=0
\end{equation} 
\begin{equation} \label{eq10}
ad(t-x)=0
\end{equation}
\begin{equation} \label{eq11}
d(b(-1+x)+y)=0
\end{equation}  
\begin{equation}\label{eq12}
d(-1+t+bz)=0
\end{equation} 
\item[(iii)] Using the relation $\sigma_2\sigma_1\tau_2=\tau_1\sigma_2\sigma_1$, we obtain the following equations.
\begin{equation} \label{eq13}
a(-1+x+cy)=0
\end{equation} 
\begin{equation}\label{eq14}
a(c(t-1)+z)=0
\end{equation} 
\begin{equation}\label{eq15}
ady=0
\end{equation} 
\begin{equation}\label{eq16}
ad(t-x)=0
\end{equation}
\begin{equation}\label{eq17}
d(c(-1+x)+z)=0
\end{equation}  
\begin{equation}\label{eq18}
d(-1+t+cy)=0
\end{equation} 
\end{itemize}

\vspace*{0.1cm}

By Theorem \ref{ThmM}, we can see that $\rho$ is one of the three representations $\rho_1, \rho_2,$ and $\rho_3$ mentioned in the theorem. We consider each case separately.
\begin{itemize}
\item[(1)] Suppose $\rho=\rho_1$. Then we have $a \neq 0, c\neq 0, b=\frac{1-a}{c}, d=0$. Using the equations (\ref{eq8}), (\ref{eq13}), and (\ref{eq14}) we get that , $y=b(1-t), x=1-cy, z=c(1-t)$. So, we have $ y=\frac{1-a}{c}(1-t), x=1-(1-a)(1-t), z=c(1-t)$, and therefore we get the required result. 
\vspace*{0.1cm}
\item[(2)] Suppose $\rho=\rho_2$. Then $a=0, c\neq 0, d\neq 0, b=\frac{1-d}{c}$. Using the equations (\ref{eq11}), (\ref{eq17}), and (\ref{eq18}) we get that , $y=b(1-x), z=c(1-x), t=1-cy$. So, we have $ y=\frac{1-d}{c}(1-x), z=c(1-x), t=1-(1-d)(1-x)$, and therefore we get the required result.
\vspace*{0.1cm}
\item[(3)] Suppose $\rho=\rho_3$. Then $a=d=0, bc\neq 0 $. Using the equations (\ref{eq4}) and (\ref{eq5}), we get that $bz-cy=0$ and $t=x$. So, $z=\frac{cy}{b}$, and therefore we get the required result. 
\end{itemize}
\end{proof}

Notice that as the Burau representation is shown to be reducible in Theorem \ref{bure}, it has been proven in \cite{102} that the homogeneous $2$-local extension of the Burau representation of $B_n$ to $SM_n$ is also reducible. This lead us to ask the following question.

\vspace*{0.1cm}

\begin{question} \label{qq}
Can we find a homogeneous $2$-local representation of $B_n$ which is reducible and its homogeneous $2$-local extension to $SM_n$ is irreducible?
\end{question}

\vspace*{0.1cm}

In the next theorem, we classify all $\Phi$-type extensions of all non-trivial homogeneous $2$-local representations of $B_n$ to $SM_n$ for all $n\geq 3$.

\begin{theorem} \label{88}
Consider $n\geq 3$ and let $\rho: B_n \rightarrow GL_n(\mathbb{C})$ be a non-trivial homogeneous $2$-local representation of $B_n$. Let $\hat{\rho}: SM_n \rightarrow M_n(\mathbb{C})$ be a $\Phi$-type extension of $\rho$ to $SM_n$. Then, $\hat{\rho}$ is equivalent to one of the following three representations.
\begin{itemize}
\item[(1)] $\hat{\rho}_1: SM_n \rightarrow M_n(\mathbb{C}) \hspace*{0.15cm} \text{such that }
\hat{\rho}_1(\sigma_i) =\rho_1(\sigma_i)$ and \\ $\hat{\rho}_1(\tau_i) =\left( \begin{array}{c|@{}c|c@{}}
   \begin{matrix}
     (u+v+w)I_{i-1} 
   \end{matrix} 
      & 0 & 0 \\
      \hline
    0 &\hspace{0.2cm} \begin{matrix}
   		ua+w & \frac{u(1-a)+v}{c}\\
   		uc+\frac{vc}{1-a} & \frac{-va}{1-a}+w 
   		\end{matrix}  & 0  \\
\hline
0 & 0 & (u+v+w)I_{n-i-1}
\end{array} \right),\\ \text{ where } a, c, u, v, w \in \mathbb{C}, a\neq 1, c \neq 0,  \hspace*{0.15cm} \text{for} \hspace*{0.2cm} 1\leq i\leq n-1$.\\
\item[(2)] $\hat{\rho}_2: SM_n \rightarrow M_n(\mathbb{C}) \hspace*{0.15cm} \text{such that } 
\hat{\rho}_2(\sigma_i)=\rho_2(\sigma_i)$ and \\ ${\rho}_2(\tau_i) =\left( \begin{array}{c|@{}c|c@{}}
   \begin{matrix}
     (u+v+w)I_{i-1} 
   \end{matrix} 
      & 0 & 0 \\
      \hline
    0 &\hspace{0.2cm} \begin{matrix}
 		\frac{-vd}{1-d}+w & \frac{u(1-d)+v}{c}\\
   		uc+\frac{vc}{1-d} & ud+w 
   		\end{matrix}  & 0  \\
\hline
0 & 0 & (u+v+w)I_{n-i-1}
\end{array} \right), \\ \text{ where } c, d, u, v, w \in \mathbb{C}, d\neq 1, c \neq 0, \hspace*{0.15cm} \text{for} \hspace*{0.15cm} 1\leq i\leq n-1$.\\
\item[(3)] $\hat{\rho}_3: SM_n \rightarrow M_n(\mathbb{C}) \hspace*{0.15cm} \text{such that }
\hat{\rho}_3(\sigma_i) =\rho_3(\sigma_i)$ and \\ $\hat{\rho}_3(\tau_i) =\left( \begin{array}{c|@{}c|c@{}}
   \begin{matrix}
     (u+v+w)I_{i-1} 
   \end{matrix} 
      & 0 & 0 \\
      \hline
    0 &\hspace{0.2cm} \begin{matrix}
   		w & ub+\frac{v}{c}\\
   		uc+\frac{v}{b} & w \\ 
   		\end{matrix}  & 0  \\
\hline
0 & 0 & (u+v+w)I_{n-i-1}
\end{array} \right), \\ \text{ where } b, c, u, v, w \in \mathbb{C}, b\neq 0, c \neq 0, \hspace*{0.15cm} \text{for} \hspace*{0.2cm} 1\leq i\leq n-1$.\\
\end{itemize}
\end{theorem}
\begin{proof}
Set $\rho(\sigma_i) =\left( \begin{array}{c|@{}c|c@{}}
   \begin{matrix}
     I_{i-1} 
   \end{matrix} 
      & 0 & 0 \\
      \hline
    0 &\hspace{0.2cm} \begin{matrix}
   		a & b\\
   		c & d
   		\end{matrix}  & 0  \\
\hline
0 & 0 & I_{n-i-1}
\end{array} \right),$ where $ad-bc\neq 0$. We have that $\hat{\rho}(\tau_i)=u\rho(\sigma_i)+v\rho^{-1}(\sigma_i)+wI_n$, where $u,v,w \in \mathbb{C}$. Direct computations give that 
$$\hat{\rho}(\tau_i) =\left( \begin{array}{c|@{}c|c@{}}
   \begin{matrix}
     (u+v+w) I_{i-1} 
   \end{matrix} 
      & 0 & 0 \\
      \hline
    0 &\hspace{0.2cm} \begin{matrix}
   		ua+\frac{vd}{ad-bc}+w & ub-\frac{vb}{ad-bc}\\
   		uc-\frac{vc}{ad-bc} & ud+\frac{va}{ad-bc}+w
   		\end{matrix}  & 0  \\
\hline
0 & 0 & (u+v+w)I_{n-i-1}
\end{array} \right).$$ 
By considering each case separately in Theorem \ref{ThmM} we get the required result.
\end{proof}

Gemein proved in \cite{10} that any homogeneous $2$-local representation $\hat{\rho}_B: SM_n \rightarrow M_n(\mathbb{Z}[t^{\pm 1}])$ that is an extension of the Burau representation $\rho_B$ of $B_n$ can be defined on the generators of $SM_n$ as follows

\vspace*{0.2cm}

$$\hat{\rho}_B(\sigma_i)=  \left( \begin{array}{c|@{}c|c@{}}
   \begin{matrix}
     I_{i-1} 
   \end{matrix} 
      & 0 & 0 \\
      \hline
    0 &\hspace{0.2cm} \begin{matrix}
   	1-t & t\\
   	1 & 0\\
\end{matrix}  & 0  \\
\hline
0 & 0 & I_{n-i-1}
\end{array} \right) \hspace*{0.2cm} \text{for} \hspace*{0.2cm} i=1,2,\ldots, n-1,$$

and

$$\hat{\rho}_B(\tau_i)=  \left( \begin{array}{c|@{}c|c@{}}
   \begin{matrix}
     I_{i-1} 
   \end{matrix} 
      & 0 & 0 \\
      \hline
    0 &\hspace{0.2cm} \begin{matrix}
   	1-t+at & t-at\\
   	1-a & a\\
\end{matrix}  & 0  \\
\hline
0 & 0 & I_{n-i-1}
\end{array} \right) \hspace*{0.2cm} \text{for} \hspace*{0.2cm} i=1,2,\ldots, n-1, \hspace*{0.2cm} \text{where} \hspace*{0.2cm} a\in \mathbb{C}.$$

\vspace*{0.1cm}

In \cite{13}, Bardakov, Chbili, and Kozlovskaya stated that the extension of the Burau representation, $\hat{\rho}_B$, to $SM_n$ can be given as follows.
$$\hat{\rho}_B(\sigma_i)=\rho_B(\sigma_i) \hspace*{0.3cm} \text{and} \hspace*{0.3cm} \hat{\rho}_B(\tau_i)=(1-a)\rho(\sigma_i)+aI_n, $$ wehre $I_n$ is the $n\times n$ identity matrix.

\vspace*{0.1cm}

The results of Gemein and of Bardakov, Chbili, and Kozlovskaya clearly show that any homogeneous $2$-local extension of the Burau representation of $B_n$ to $SM_n$ is of $\Phi$-type. This lead us to ask the following question.

\vspace*{0.1cm}

\begin{question}
If $\rho$ is a homogeneous $k$-local representation of the braid group $B_n$, then what is the relation between $k$-local extensions and $\Phi$-type extensions of $\rho$ to the singular braid monoid $SM_n$?
\end{question}

\vspace*{0.1cm}

\section{Homogeneous $3$-Local Representations of $B_n$ and its Homogeneous $3$-Local Extensions to $SM_n$} 

\vspace*{0.1cm}

In this section, we extend the result of Mikhalchishina and the result in Section 3 in two steps. First, we classify all non-trivial homogeneous $3$-local representations of $B_n$ for all $n\geq 4$. Second, we classify all non-trivial homogeneous $3$-local extensions of these representations to the singular braid monoid $SM_n$ for all $n\geq 4$.

\vspace*{0.1cm}

\begin{theorem}\label{repBn}
For $n\geq 4$, let $\nu :B_n\to GL_{n+1}(\mathbb{C})$ be a non-trivial homogeneous $3$-local representation of $B_n$. Then $\nu$ is equivalent to one of the following eight representations $\nu_j$, $1\leq j\leq 8$, where $\nu_j(\sigma_i)=\left( \begin{array}{c|@{}c|c@{}}
   \begin{matrix}
     I_{i-1} 
   \end{matrix} 
      & 0 & 0 \\
      \hline
    0 &\hspace{0.2cm} \begin{matrix}
   		M_j
   		\end{matrix}  & 0  \\
\hline
0 & 0 & I_{n-i-1}
\end{array} \right),
\text{ for all } 1\leq i \leq n-1 \text{ with } M_j \text{ is given by }$
\begin{enumerate}
\item $M_1=\begin{pmatrix}
     1& 0 & 0\\
     0 & m_{22} & m_{23}\\
     0 & \frac{1-m_{22}}{m_{23}} & 0
   \end{pmatrix} $, where $m_{22}\neq1$ and $m_{23}\neq0.$
   \vspace*{0.1cm}
\item  $M_2=\begin{pmatrix}
     0 &m_{12} & 0\\
     \frac{1-m_{22}}{m_{12}} & m_{22} & 0\\
     0 & 0 & 1
   \end{pmatrix} $, where $m_{22}\neq1$ and $m_{12}\neq0$.
      \vspace*{0.1cm}
\item $M_3=\begin{pmatrix}
     1 &-\frac{m_{22}}{m_{32}} & 0\\
     0 &m_{22} & 0\\
     0 & m_{32} & 1
   \end{pmatrix} $, where $m_{22}m_{32}\neq0$.
      \vspace*{0.1cm}
\item  $M_4=\begin{pmatrix}
     1 & 0 & 0\\
     -\frac{m_{22}}{m_{32}}& m_{22} & m_{23}\\
     0 & 0 & 1
   \end{pmatrix} $, where $m_{23}m_{22}\neq0$.
      \vspace*{0.1cm}
\item $M_5=\begin{pmatrix}
     1 & 0 & 0\\
     0 & 0 & m_{23}\\
     0 & m_{32} & 1-m_{23}m_{32}
   \end{pmatrix} $, where $m_{23}m_{32}\neq0$.
      \vspace*{0.1cm}
\item $M_6=\begin{pmatrix}
     1-m_{12}m_{21} & m_{12} & 0\\
     m_{21} & 0 & 0\\
     0 & 0 & 1
   \end{pmatrix} $, where $m_{12}m_{21}\neq0$.
      \vspace*{0.1cm}
\item  $M_7=\begin{pmatrix}
     1 & 0 & 0\\
     0 & 0 & m_{23}\\
     0 & m_{32} & 0
   \end{pmatrix},$ where $m_{23}m_{32}\neq0.$
         \vspace*{0.1cm}
    \item  $M_8=\begin{pmatrix}
     0 & m_{12} & 0\\
      m_{21} & 0 & 0\\
       0 & 0 & 1
\end{pmatrix},$ where $m_{12}m_{21}\neq0.$
\end{enumerate} 
\end{theorem}
\begin{proof}
Since $\nu$ is a homogeneous $3$-local representation, it follows that there exists a $3\times 3$ complex matrix  $M=\begin{pmatrix}
     m_{11}&m_{12} & m_{13}\\
     m_{21}&m_{22} & m_{23}\\
     m_{31} & m_{32} & m_{33}
   \end{pmatrix}$
such that 
$\nu(\sigma_i)=\left( \begin{array}{c|@{}c|c@{}}
   \begin{matrix}
     I_{i-1} 
   \end{matrix} 
      & 0 & 0 \\
      \hline
    0 &\hspace{0.2cm} \begin{matrix}
   		M
   		\end{matrix}  & 0  \\
\hline
0 & 0 & I_{n-i-1}
\end{array} \right),$ for all $1\leq i\leq n-1$. Note that $\det(\nu(\sigma_i))=\det(M)$, which gives that $\det(M)\neq 0$.\\

First of all, let's check the relation $\nu(\sigma_i\sigma_j)=\nu(\sigma_j\sigma_i)$ for $|i-j|>1$. Without loss of generality, it suffices to consider two cases: $j>i+2$ and $j=i+2$.\\

For the case $j>i+2$, we have 
$$\nu(\sigma_i\sigma_j)=\nu(\sigma_j\sigma_i)=\left( \begin{array}{c|c|c|c|c}
   \begin{matrix}
     I_{i-1} 
   \end{matrix} 
      & 0 & 0 &0 &0\\
      \hline
    0 & \begin{matrix}
    M 
    \end{matrix} & 0 &0 &0 \\
\hline
0 & 0 & \begin{matrix}
I_{j-i-3}
\end{matrix} &0 &0\\
\hline
0 & 0 & 0 & \begin{matrix}
M
\end{matrix}&0\\
\hline
0 & 0 & 0 & 0&  \begin{matrix}
I_{n-j-1}
\end{matrix}
\end{array} \right).$$

For the case $j=i+2$, direct computations give that
   $$\nu(\sigma_i\sigma_j)=\left( \begin{array}{c|@{}c|c@{}}
   \begin{matrix}
     I_{i-1} 
   \end{matrix} 
      & 0 & 0 \\
      \hline
    0 &\hspace{0.2cm} \begin{matrix}
   		S_1
   		\end{matrix}  & 0  \\
\hline
0 & 0 & I_{n-i-3}
\end{array} \right) 
\text{ and  } 
\nu(\sigma_j\sigma_i)=\left( \begin{array}{c|@{}c|c@{}}
   \begin{matrix}
     I_{i-1} 
   \end{matrix} 
      & 0 & 0 \\
      \hline
    0 &\hspace{0.2cm} \begin{matrix}
   		S_2
   		\end{matrix}  & 0  \\
\hline
0 & 0 & I_{n-i-3}
\end{array} \right), 
$$ where $$S_1=\left(
\begin{array}{ccccc}
 m_{11} & m_{12} & m_{11} m_{13} & m_{12}m_{13} & m_{13}^2 \\
 m_{21} & m_{22} & m_{11}m_{23} & m_{12}m_{23} & m_{13}m_{23} \\
 m_{31} & m_{32} & m_{11}m_{33} & m_{12}m_{33} & m_{13}m_{33} \\
 0 & 0 & m_{21} & m_{22} & m_{23} \\
 0 & 0 & m_{31} & m_{32} & m_{33} \\
\end{array}
\right)$$ and 

$$S_2=\left(
\begin{array}{ccccc}
 m_{11} & m_{12} & m_{13} & 0 & 0 \\
 m_{21} & m_{22} & m_{23} & 0 & 0 \\
 m_{11}m_{31} & m_{11}m_{32} & m_{11}m_{33} & m_{12} & m_{13} \\
 m_{21}m_{31} & m_{21}m_{32} & m_{21}m_{33} & m_{22} & m_{23} \\
 m_{31}^2 & m_{31}m_{32} & m_{31}m_{33} & m_{32} & m_{33} \\
\end{array}
\right).$$\\

Hence, the relation $\sigma_i\sigma_j=\sigma_j\sigma_i$ yields $S_1=S_2$. This implies that $m_{13}=m_{31}=0$. Then we get the following equations.
\begin{equation}\label{B301}
m_{23} - m_{11} m_{23}=0
\end{equation}
\begin{equation}\label{N302}
m_{12} m_{23}=0
\end{equation}
\begin{equation}\label{N303}
-m_{32} + m_{11} m_{32}=0
\end{equation}
\begin{equation}\label{N304}
m_{12} - m_{12} m_{33}=0
\end{equation}
\begin{equation}\label{N305}
m_{21} m_{32}=0
\end{equation}
\begin{equation}\label{N306}
-m_{21} + m_{21} m_{33}=0
\end{equation}

The relation $\sigma_{i}\sigma_{i+1}\sigma_{i}=\sigma_{i+1}\sigma_{i}\sigma_{i+1}$ yields the following.

$$\left( \begin{array}{c|@{}c|c@{}}
   \begin{matrix}
     I_{i-1} 
   \end{matrix} 
      & 0 & 0 \\
      \hline
    0 &\hspace{0.2cm} \begin{matrix}
   		T_1
   		\end{matrix}  & 0  \\
\hline
0 & 0 & I_{n-i-2}
\end{array} \right)=
\left( \begin{array}{c|@{}c|c@{}}
   \begin{matrix}
     I_{i-1} 
   \end{matrix} 
      & 0 & 0 \\
      \hline
    0 &\hspace{0.2cm} \begin{matrix}
   		T_2
   		\end{matrix}  & 0  \\
\hline
0 & 0 & I_{n-i-2}
\end{array} \right), $$
\vspace*{0.2cm}
where $T_1=(t_{11},t_{12},t_{13},t_{14})\text{ and } T_2=(t_{21},t_{22},t_{23},t_{24}),$
with  
$$t_{11}=\left(
\begin{array}{c}
m_{11} (m_{11}+m_{12}m_{21}) \\
 m_{21} (m_{11} +m_{11}m_{22}+m_{21} m_{23}) \\
 m_{21} (m_{11}m_{32}+m_{21}m_{33})\\
 0 \\
\end{array}
\right),$$
\vspace*{0.1cm}
$$ t_{12}=\left(
\begin{array}{c}
m_{12} (m_{11}+m_{11}m_{22}+m_{12} m_{32}) \\
m_{22} (m_{11} m_{22}+m_{23} (m_{21}+m_{32}))+m_{12} (m_{21}+m_{22} m_{32}) \\
m_{11} m_{22} m_{32}+m_{12} m_{32}^2+m_{22} m_{33} (m_{21}+m_{32}) \\
m_{32}^2 \\
\end{array}
\right),$$
\vspace*{0.1cm}
$$t_{13}=\left(
\begin{array}{c}
 m_{12} (m_{11} m_{23}+m_{12} m_{33}) \\
 m_{11} m_{22} m_{23}+m_{22} m_{33} (m_{12}+m_{23})+m_{21} m_{23}^2\\
  m_{11} m_{23} m_{32}+m_{33} (m_{12} m_{32}+m_{21}m_{23}+m_{22} m_{33}) \\
 m_{32} m_{33} \\
\end{array}
\right),$$
\vspace*{0.1cm}
$$t_{14}=\left(
\begin{array}{c}
 0 \\
 m_{23}^2 \\
 m_{23} m_{33} \\
m_{33} \\
\end{array}
\right),$$
\vspace*{0.1cm}
$$
t_{21}=\left(
\begin{array}{c}
 m_{11} \\
 m_{11} m_{21} \\
 m_{21}^2  \\
 0  \\
\end{array}
\right),$$
\vspace*{0.1cm}
$$
t_{22}=\left(
\begin{array}{cc}
 m_{11} m_{12} \\
 m_{11}^2 m_{22}+m_{11} m_{12} m_{32}+m_{11} m_{21} m_{23}+m_{12} m_{21} m_{33} \\
  m_{11} m_{22} (m_{21}+m_{32})+m_{21} (m_{21} m_{23}+m_{22} m_{33}) \\
 m_{32} (m_{11} m_{32}+m_{21} m_{33})  \\
\end{array}
\right),$$
\vspace*{0.1cm}
$$
t_{23}=\left(
\begin{array}{cc}
 m_{12}^2  \\
 m_{11} m_{22} (m_{12}+m_{23})+m_{12} (m_{12} m_{32}+m_{22} m_{33}) \\
m_{12} m_{22} (m_{21}+m_{32})+m_{21} m_{22} m_{23}+m_{22}^2 m_{33}+m_{23} m_{32} \\
m_{32} (m_{12} m_{32}+m_{22} m_{33}+m_{33}) \\
\end{array}
\right),$$ and 
\vspace*{0.1cm}
$$t_{24}=\left(
\begin{array}{cc}
 0 \\
m_{23} (m_{11} m_{23}+m_{12} m_{33}) \\
m_{23} (m_{21} m_{23}+m_{22} m_{33}+m_{33}) \\
m_{33} (m_{23} m_{32}+m_{33}) \\
\end{array}
\right).$$
Then, $T_1=T_2$ leads to the following equations.
\begin{equation}\label{B307}
m_{11}^2+m_{11} m_{12} m_{21}-m_{11}=0
\end{equation}
\begin{equation}\label{B308}
m_{11} m_{12} m_{22}+m_{12}^2 m_{32}=0
\end{equation}
\begin{equation}\label{B309}
-m_{12}^2 + m_{11} m_{12} m_{23} + m_{12}^2 m_{33}=0
\end{equation}
\begin{equation}\label{B310}
m_{21} (m_{11} m_{22}+m_{21} m_{23})=0
\end{equation}
\begin{equation}\label{B311}
\begin{aligned}
& m_{12} m_{21} + m_{22} (m_{11} m_{22} + m_{21} m_{23}) + (m_{12} m_{22} + m_{22} m_{23}) m_{32} \\
& -m_{11} (m_{11} m_{22} + m_{12} m_{32}) - m_{21} (m_{11} m_{23} + m_{12} m_{33})=0
\end{aligned}
\end{equation}
\begin{equation}\label{B312}
\begin{aligned}
& m_{23} (m_{11} m_{22} + m_{21} m_{23}) - 
m_{12} (m_{11} m_{22} + m_{12} m_{32})\\
& + (m_{12} m_{22} + m_{22} m_{23}) m_{33}
-m_{22} (m_{11} m_{23} + m_{12} m_{33})=0
\end{aligned}
\end{equation}
\begin{equation}\label{B313}
m_{23}^2 - m_{23} (m_{11} m_{23} + m_{12} m_{33})=0
\end{equation}
\begin{equation}\label{B314}
-m_{21}^2 + m_{21} (m_{11} m_{32} + m_{21} m_{33})=0
\end{equation}
\begin{equation}\label{B315}
\begin{aligned}
& -m_{11} (m_{21} m_{22} + m_{22} m_{32}) + m_{22} (m_{11} m_{32} + m_{21} m_{33})\\
& -m_{21} (m_{21} m_{23} + m_{22} m_{33}) + m_{32} (m_{12} m_{32} + m_{22} m_{33})=0
\end{aligned}
\end{equation}
\begin{equation}\label{B316}
\begin{aligned}
&-m_{23} m_{32} - m_{12} (m_{21} m_{22} + m_{22} m_{32}) + m_{23} (m_{11} m_{32} + m_{21} m_{33})\\
 & - m_{22} (m_{21} m_{23} + m_{22} m_{33}) + m_{33} (m_{12} m_{32} + m_{22} m_{33})=0
\end{aligned}
\end{equation}
\begin{equation}\label{B317}
-m_{23} (m_{21} m_{23} + m_{22} m_{33})=0
\end{equation}
\begin{equation}\label{B318}
\begin{aligned}
& -m_{11} (m_{21} m_{22} + m_{22} m_{32}) + m_{22} (m_{11} m_{32} + m_{21} m_{33}) \\
& -m_{21} (m_{21} m_{23} + m_{22} m_{33}) + m_{32} (m_{12} m_{32} + m_{22} m_{33})=0
\end{aligned}
\end{equation}
\begin{equation}\label{B319}
-m_{23} (m_{21} m_{23} + m_{22} m_{33})=0
\end{equation}
\begin{equation}\label{B320}
m_{32}^2 - m_{11} m_{32}^2 - m_{21} m_{32} m_{33}=0
\end{equation}
\begin{equation}\label{B321}
-m_{12} m_{32}^2 - m_{22} m_{32} m_{33}=0
\end{equation}
\begin{equation}\label{B322}
m_{33} - m_{23} m_{32} m_{33} - m_{33}^2=0
\end{equation}

Taking into consideration that $\det(M)\neq0$ and using Mathematica software, we solve the system formed by the equations: Equation(\ref{B301}) through Equation(\ref{B322}). We obtain the following nine solutions.
\begin{itemize}
\item[(0)] $ m_{11}=1, m_{12}=0, m_{21}=0, m_{22}=1, m_{23}=0, m_{32}= 0, m_{33}=1$ with $m_{22}\neq1$ and $m_{12}\neq0$. The matrix  $M_0=\begin{pmatrix}
     1 & 0 & 0\\
      0 & 1 & 0\\
       0 & 0 & 1
\end{pmatrix}=I_3$.
This case leads to the trivial representation of $B_n$, which is rejected since $\nu$ is assumed to be non-trivial.

\item[(1)] $m_{11}=1, m_{12}=0, m_{21}=0, m_{32}=\frac{1 - m_{22}}{m_{23}}, m_{33}=0$ with $m_{22}\neq1$ and $m_{23}\neq0$. The matrix $M_1=\begin{pmatrix}
     1& 0 & 0\\
     0 & m_{22} & m_{23}\\
     0 & \frac{1-m_{22}}{m_{23}} & 0
   \end{pmatrix} $, where $m_{22}\neq1$ and $m_{23}\neq0,$ as required.
\vspace{0.1cm}
\item[(2)] $m_{11}=0, m_{21}=\frac{1-m_{22}}{m_{12}}, m_{23}=0, m_{32}=0, m_{33}=1$ with $m_{22}\neq1$ and $m_{12}\neq0$. The matrix $M_2=\begin{pmatrix}
     0 &m_{12} & 0\\
     \frac{1-m_{22}}{m_{12}} & m_{22} & 0\\
     0 & 0 & 1
   \end{pmatrix} $, where $m_{22}\neq1$ and $m_{12}\neq0$, as required.
\vspace{0.1cm}
\item[(3)] $m_{11}=1, m_{12}=-\frac{m_{22}}{m_{32}}, m_{21}=0, m_{23}= 0, m_{33}=1$ with $m_{22}m_{32}\neq0$. The matrix $M_3=\begin{pmatrix}
     1 &-\frac{m_{22}}{m_{32}} & 0\\
     0 &m_{22} & 0\\
     0 & m_{32} & 1
   \end{pmatrix} $, where $m_{22}m_{32}\neq0$, as required.
\vspace{0.1cm}
\item[(4)] $m_{11}=1, m_{12}=0, m_{21}=-\frac{m_{22}}{m_{23}}, m_{32}=0, m_{33}=1$ with $m_{22}m_{23}\neq0$. The matrix $M_4=\begin{pmatrix}
     1 & 0 & 0\\
     -\frac{m_{22}}{m_{32}}& m_{22} & m_{23}\\
     0 & 0 & 1
   \end{pmatrix} $, where $m_{23}m_{22}\neq0$, as required.
\vspace{0.1cm}
\item[(5)] $m_{11}=1, m_{12}=0, m_{21}=0, m_{22}=0, m_{33}=1-m_{23} m_{32}$ with $m_{32}m_{23}\neq0$. The matrix $M_5=\begin{pmatrix}
     1 & 0 & 0\\
     0 & 0 & m_{23}\\
     0 & m_{32} & 1-m_{23}m_{32}
   \end{pmatrix} $, where $m_{23}m_{32}\neq0$, as required.
 \vspace{0.1cm}
\item[(6)] $m_{11}=1-m_{12} m_{21}, m_{22}=0, m_{23}=0, m_{32}=0, m_{33}=1$ with $m_{12}m_{21}\neq0$. The matrix $M_6=\begin{pmatrix}
     1-m_{12}m_{21} & m_{12} & 0\\
     m_{21} & 0 & 0\\
     0 & 0 & 1
   \end{pmatrix} $, where $m_{12}m_{21}\neq0$, as required.
\vspace{0.1cm}
\item[(7)] $m_{11}=1, m_{12}=0, m_{21}=0, m_{22}= 0, m_{33}=0$ with $m_{32}m_{23}\neq0$. The matrix $M_7=\begin{pmatrix}
     1 & 0 & 0\\
     0 & 0 & m_{23}\\
     0 & m_{32} & 0
   \end{pmatrix},$ where $m_{23}m_{32}\neq0$, as required.
\vspace{0.1cm}
\item[(8)] $m_{11}=0, m_{22}=0, m_{23}=0, m_{32}=0, m_{33}=1$ with $m_{12}m_{21}\neq0$. The matrix  $M_8=\begin{pmatrix}
     0 & m_{12} & 0\\
      m_{21} & 0 & 0\\
       0 & 0 & 1
\end{pmatrix},$ where $m_{12}m_{21}\neq0$, as required.
\end{itemize}
\end{proof}

\vspace*{0.1cm}

We can easily see that for all $n\geq 4$, the complex specialization of the $F$-representation defined in Definition \ref{Fdef} is a special representation of the third family of representations in Theorem \ref{repBn}. That is, $\rho_F=\nu_3$ in the case $m_{22}=-t$ and $m_{32}=t$.

\vspace*{0.1cm}

Now, we extend the result of the previous theorem to the singular braid monoid $SM_n$ for $n\geq 4$.

\vspace*{0.1cm}

\begin{theorem} \label{repSMn}
For $n\geq 4$, let $\nu: B_n\to GL_{n+1}(\mathbb{C})$ be a non-trivial homogeneous $3$-local representation of $B_n$.
Then, any extension $\nu':SM_n\to M_{n+1}(\mathbb{C})$ of $\nu$ to $SM_n$ is equivalent to one of the following representations $\nu'_j$, $1\leq j\leq 8$, where
$$\nu'_{j}(\sigma_i)=\nu(\sigma_i)=\left( \begin{array}{c|@{}c|c@{}}
   \begin{matrix}
     I_{i-1} 
   \end{matrix} 
      & 0 & 0 \\
      \hline
    0 &\hspace{0.2cm} \begin{matrix}
   		M_j
   		\end{matrix}  & 0  \\
\hline
0 & 0 & I_{n-i-1}
\end{array} \right),$$
and
$$\nu'_{j}(\tau_i)=
\left( \begin{array}{c|@{}c|c@{}}
   \begin{matrix}
     I_{i-1} 
   \end{matrix} 
      & 0 & 0 \\
      \hline
    0 &\hspace{0.2cm} \begin{matrix}
   		N_j
   		\end{matrix}  & 0  \\
\hline
0 & 0 & I_{n-i-1}
\end{array} \right),
 $$
for all  $1\leq i \leq n-1$, with the matrices $M_j$ and $N_j$ are given as follows.
\begin{enumerate}
\item $M_1=\begin{pmatrix}
     1& 0 & 0\\
     0 & m_{22} & m_{23}\\
     0 & \frac{1-m_{22}}{m_{23}} & 0
   \end{pmatrix} $, where $m_{22}\neq1, m_{23}\neq0$ and 
   \vspace*{0.1cm}
   \begin{enumerate}
   \item $N_1=\begin{pmatrix}
     1& 0 & 0\\
     0 & n_{22} & m_{23}^2n_{32}\\
     0 & n_{32} & n_{22}
   \end{pmatrix} $ if $m_{22}=0$.
      \vspace*{0.1cm}
   \item $N_1=\begin{pmatrix}
     1& 0 & 0\\
     0 & 1-m_{23}n_{32}& \frac{-m_{23}^2n_{32}}{m_{22}-1}\\
     0 & n_{32} & \frac{-1+m_{22}+m_{23}n_{32}}{m_{22}-1}
   \end{pmatrix}$ otherwise.
   \end{enumerate}
      \vspace*{0.1cm}
   \item $M_2=\begin{pmatrix}
     0 &m_{12} & 0\\
     \frac{1-m_{22}}{m_{12}} & m_{22} & 0\\
     0 & 0 & 1
   \end{pmatrix} $, where $m_{22}\neq1, m_{12}\neq0$ and 
   \vspace*{0.1cm} 
   \begin{enumerate}
   \item $N_2=\begin{pmatrix}
     n_{11}& m_{12}^2n_{21} & 0\\
     n_{21} & n_{11} & 0\\
     0 & 0 & 1
   \end{pmatrix} $ if $m_{22}=0$.
      \vspace*{0.1cm}
   \item $N_2=\left(
\begin{array}{ccc}
 \frac{m_{12} n_{21}+m_{22}-1}{m_{22}-1} & -\frac{m_{12}^2 n_{21}}{m_{22}-1} & 0 \\
 n_{21} & 1-m_{12} n_{21} & 0 \\
 0 & 0 & 1 \\
\end{array}
\right) $ otherwise.
   \end{enumerate}   
      \vspace*{0.1cm}
   \item $M_3=\begin{pmatrix}
     1 &\frac{-m_{22}}{m_{32}} & 0\\
     0 & m_{22} & 0\\
     0 & m_{32} & 1
   \end{pmatrix} $, where $m_{22}m_{32}\neq0$ and \\ 
      \vspace*{0.1cm}
   $N_3=\begin{pmatrix}
     1& n_{12} & 0\\
     0 & 1-m_{32}n_{12}+\frac{m_{32}n_{12}}{m_{22}} & 0\\
     0 & -\frac{m_{32}^2n_{12}}{m_{22}} & 1
   \end{pmatrix} $.
     \vspace*{0.1cm}
\item $M_4=\begin{pmatrix}
     1 & 0 & 0\\
     -\frac{m_{22}}{m_{32}}& m_{22} & m_{23}\\
     0 & 0 & 1
   \end{pmatrix} $, where $m_{23}m_{22}\neq0$ and \\ 
      \vspace*{0.1cm}
   $N_4=\begin{pmatrix}
     1& 0 & 0\\
     n_{21} & 1-m_{23}n_{21}+\frac{m_{23}n_{21}}{m_{22}} & -\frac{m_{23}^2n_{21}}{m_{22}}\\
     0 & 0 & 1
   \end{pmatrix} $. 
     \vspace*{0.1cm}
\item $M_5=\begin{pmatrix}
     1 & 0 & 0\\
     0 & 0 & m_{23}\\
     0 & m_{32} & 1-m_{23}m_{32}
   \end{pmatrix} $, where $m_{23}m_{32}\neq0$ and
     \begin{enumerate}
        \vspace*{0.1cm}
   \item $N_5=\begin{pmatrix}
     1 & 0 & 0\\
     0 & n_{22} & n_{23}\\
     0 & \frac{n_{23}}{m_{23}^2} & n_{22}
   \end{pmatrix}$ if $m_{32}m_{23}=1$.
      \vspace*{0.1cm}
   \item $N_5=\left(
\begin{array}{ccc}
 1 & 0 & 0 \\
 0 & n_{22} & m_{23}-m_{23} n_{22} \\
 0 & m_{32}-m_{32} n_{22} & m_{23} m_{32} (n_{22}-1)+1 \\
\end{array}
\right)$ otherwise.
   \vspace*{0.1cm}
   \end{enumerate}
\item $M_6=\begin{pmatrix}
     1-m_{12}m_{21} & m_{12} & 0\\
     m_{21} & 0 & 0\\
     0 & 0 & 1
   \end{pmatrix} $ where $m_{12}m_{21}\neq0$ and
      \vspace*{0.1cm}
   \begin{enumerate}
   \item $N_6=\begin{pmatrix}
     n_{11} & n_{12} & 0\\
     \frac{n_{12}}{m_{12}^2} & n_{11} & 0\\
     0 & 0 & 1
   \end{pmatrix} $ if $m_{12}m_{21}=1$.
      \vspace*{0.1cm}
   \item $N_6=\left(
\begin{array}{ccc}
 m_{12} m_{21} (n_{22}-1)+1 & m_{12}-m_{12} n_{22} & 0 \\
 m_{21}-m_{21} n_{22} & n_{22} & 0 \\
 0 & 0 & 1 \\
\end{array}
\right)$ otherwise.
   \vspace*{0.1cm}
   \end{enumerate}
\item $M_7=\begin{pmatrix}
     1 & 0 & 0\\
     0 & 0 & m_{23}\\
     0 & m_{32} & 0
   \end{pmatrix} $ where $m_{23}m_{32}\neq0$ and \\
      \vspace*{0.1cm}
    $N_7=\left(
\begin{array}{ccc}
 1 & 0 & 0 \\
 0 & n_{22} & n_{23} \\
 0 & \frac{m_{32}n_{23}}{m_{23}} & n_{22} \\
\end{array}
\right).$
   \vspace*{0.1cm}
\item $M_8=\begin{pmatrix}
     0 &m_{12} & 0\\
     m_{21} & 0 & 0\\
     0 & 0 & 1
   \end{pmatrix} $ where $m_{12}m_{21}\neq1$ and \\
      \vspace*{0.1cm}
    $N_8=\left(
\begin{array}{ccc}
 1 & 0 & 0 \\
 0 & n_{22} & n_{23} \\
 0 & \frac{m_{32}n_{23}}{m_{23}} & n_{22} \\
\end{array}
\right).$
\end{enumerate}
\noindent Moreover, if $\nu'(\tau_i)$ is invertible for all $1\leq i\leq n-1$, then $\nu'$ is a representation of $SB_n$.
\end{theorem}
\begin{proof}
Since $\nu':SM_n\to M_{n+1}(\mathbb{C})$ is the extension of a representation $\nu$ of $B_n$, it follows that $\nu'(\sigma_i)=\nu(\sigma_i)=\left( \begin{array}{c|@{}c|c@{}}
   \begin{matrix}
     I_{i-1} 
   \end{matrix} 
      & 0 & 0 \\
      \hline
    0 &\hspace{0.2cm} \begin{matrix}
   		M
   		\end{matrix}  & 0  \\
\hline
0 & 0 & I_{n-i-1}
\end{array} \right),
\text{ for all } 1\leq i\leq n-1,$
where $M=\begin{pmatrix}
     m_{11}& m_{12} & 0\\
     m_{21}&m_{22} & m_{23}\\
     0 & m_{32} & m_{33}
   \end{pmatrix}$ is the matrix defined by Theorem \ref{repBn} for all $1\leq i\leq n-1$. Set $\nu(\tau_i)=\left( \begin{array}{c|@{}c|c@{}}
   \begin{matrix}
     I_{i-1} 
   \end{matrix} 
      & 0 & 0 \\
      \hline
    0 &\hspace{0.2cm} \begin{matrix}
   		N
   		\end{matrix}  & 0  \\
\hline
0 & 0 & I_{n-i-1}
\end{array} \right),
\text{ for all } 1\leq i\leq n-1,$ where $N=\begin{pmatrix}
     n_{11}&n_{12} & n_{13}\\
     n_{21}&n_{22} & n_{23}\\
     n_{31} & n_{32} & n_{33}
   \end{pmatrix}\in M_3(\mathbb{C}).$ Without loss of generality, we restrict the study of the relation $\tau_i\tau_j=\tau_j\tau_i$ for $|i-j|\geq 2$ to two cases: $j>i+2$ and $j=i+2$.\\
   
For the case $j>i+2$ we have
$$\nu'(\tau_i\tau_j)=\nu'(\tau_j\tau_i)=\left( \begin{array}{c|c|c|c|c}
   \begin{matrix}
     I_{i-1} 
   \end{matrix} 
      & 0 & 0 &0 &0\\
      \hline
    0 & \begin{matrix}
    N 
    \end{matrix} & 0 &0 &0 \\
\hline
0 & 0 & \begin{matrix}
I_{j-i-3}
\end{matrix} &0 &0\\
\hline
0 & 0 & 0 & \begin{matrix}
N
\end{matrix}&0\\
\hline
0 & 0 & 0 & 0&  \begin{matrix}
I_{n-j-1}
\end{matrix}
\end{array} \right).$$

For the case $j=i+2$, direct computations show that
$$\nu(\tau_i\tau_j)=\left( \begin{array}{c|@{}c|c@{}}
   \begin{matrix}
     I_{i-1} 
   \end{matrix} 
      & 0 & 0 \\
      \hline
    0 &\hspace{0.2cm} \begin{matrix}
   		U_1
   		\end{matrix}  & 0  \\
\hline
0 & 0 & I_{n-i-3}
\end{array} \right) 
\text{ and } 
\nu(\tau_j\tau_i)=\left( \begin{array}{c|@{}c|c@{}}
   \begin{matrix}
     I_{i-1} 
   \end{matrix} 
      & 0 & 0 \\
      \hline
    0 &\hspace{0.2cm} \begin{matrix}
   		U_2
   		\end{matrix}  & 0  \\
\hline
0 & 0 & I_{n-i-3}
\end{array} \right) 
$$ 
where 
$$U_1=\left(
\begin{array}{ccccc}
 n_{11} & n_{12} & n_{11}n_{13} & n_{12}n_{13} & n_{13}^2 \\
 n_{21} & n_{22} & n_{11}n_{23} & n_{12}n_{23} & n_{13}n_{23} \\
 n_{31} & n_{32} & n_{11}n_{33} & n_{12}n_{33} & n_{13}n_{33} \\
 0 & 0 & n_{21} & n_{22} & n_{23} \\
 0 & 0 & n_{31} & n_{32} & n_{33} \\
\end{array}
\right)$$ 
and 
$$U_2=\left(
\begin{array}{ccccc}
 n_{11} & n_{12} & n_{13} & 0 & 0 \\
 n_{21} & n_{22} & n_{23} & 0 & 0 \\
 n_{11}n_{31} & n_{11}n_{32} & n_{11}n_{33} & n_{12} & n_{13} \\
 n_{21}n_{31} & n_{21}n_{32} & n_{21}n_{33} & n_{22} & n_{23} \\
 n_{31}^2 & n_{31}n_{32} & n_{31}n_{33} & n_{32} & n_{33} \\
\end{array}
\right).$$

The relation $\tau_i\tau_j=\tau_j\tau_i$ yields $U_1=U_2$ and this implies that $n_{13}=n_{31}=0$. Using this result we get the following equations.
\begin{equation}\label{S301}
-n_{23} + n_{11}n_{23}=0
\end{equation}
\begin{equation}\label{S302}
n_{12}n_{23}=0
\end{equation}
\begin{equation}\label{S303}
n_{32}- n_{11}n_{32}=0
\end{equation}
\begin{equation}\label{S304}
-n_{12} + n_{12} n_{33}=0
\end{equation}
\begin{equation}\label{S305}
n_{21} n_{32}=0
\end{equation}
\begin{equation}\label{S306}
n_{21} - n_{21} n_{33}=0
\end{equation}

We restrict the study of the relation $\tau_i\sigma_j=\sigma_j\tau_i$ for $|i-j|\geq 2$ to 4 cases: (i) $j>i+2$, (ii) $i>j+2$, (iii) $j=i+2$ and (iv) $i=j+2$.\\

For the case (i) $j>i+2$ we have
$$\nu'(\tau_i\sigma_j)=\nu'(\sigma_j\tau_i)=\left( \begin{array}{c|c|c|c|c}
   \begin{matrix}
     I_{i-1} 
   \end{matrix} 
      & 0 & 0 &0 &0\\
      \hline
    0 & \begin{matrix}
    M 
    \end{matrix} & 0 &0 &0 \\
\hline
0 & 0 & \begin{matrix}
I_{j-i-3}
\end{matrix} &0 &0\\
\hline
0 & 0 & 0 & \begin{matrix}
N
\end{matrix}&0\\
\hline
0 & 0 & 0 & 0&  \begin{matrix}
I_{n-j-1}
\end{matrix}
\end{array} \right).$$

For the case (ii) $i>j+2$ we have
$$\nu'(\tau_i\sigma_j)=\nu'(\sigma_j\tau_i)=\left( \begin{array}{c|c|c|c|c}
   \begin{matrix}
     I_{i-1} 
   \end{matrix} 
      & 0 & 0 &0 &0\\
      \hline
    0 & \begin{matrix}
    N 
    \end{matrix} & 0 &0 &0 \\
\hline
0 & 0 & \begin{matrix}
I_{j-i-3}
\end{matrix} &0 &0\\
\hline
0 & 0 & 0 & \begin{matrix}
M
\end{matrix}&0\\
\hline
0 & 0 & 0 & 0&  \begin{matrix}
I_{n-j-1}
\end{matrix}
\end{array} \right).$$

For the case (iii) $j=i+2$, direct computations give that
   $$\nu(\tau_i\sigma_j)=\left( \begin{array}{c|@{}c|c@{}}
   \begin{matrix}
     I_{i-1} 
   \end{matrix} 
      & 0 & 0 \\
      \hline
    0 &\hspace{0.2cm} \begin{matrix}
   		V_1
   		\end{matrix}  & 0  \\
\hline
0 & 0 & I_{n-i-3}
\end{array} \right) 
\text{ and } 
\nu(\sigma_j\tau_i)=\left( \begin{array}{c|@{}c|c@{}}
   \begin{matrix}
     I_{i-1} 
   \end{matrix} 
      & 0 & 0 \\
      \hline
    0 &\hspace{0.2cm} \begin{matrix}
   		V_2
   		\end{matrix}  & 0  \\
\hline
0 & 0 & I_{n-i-3}
\end{array} \right), 
$$
where 
$$V_1=\left(
\begin{array}{ccccc}
 n_{11} & n_{12} & 0 & 0 & 0 \\
 n_{21} & n_{22} & m_{11}n_{23} & m_{12}n_{23} & 0 \\
 0 & n_{32} & m_{11}n_{33} & m_{12}n_{33} & 0 \\
 0 & 0 & m_{21} & m_{22} & m_{23} \\
 0 & 0 & 0 & m_{32} & m_{33} \\
\end{array}
\right)$$
and 
$$V_2=\left(
\begin{array}{ccccc}
 n_{11} & n_{12} & 0 & 0 & 0 \\
 n_{21} & n_{22} & n_{23} & 0 & 0 \\
 0 & m_{11}n_{32} & m_{11}n_{33} & m_{12} & 0 \\
 0 & m_{21}n_{32} & m_{21}n_{33} & m_{22} & m_{23} \\
 0 & 0 & 0 & m_{32} & m_{33} \\
\end{array}
\right).$$

\vspace*{0.1cm}

Hence, the relation $\tau_i\sigma_j=\sigma_j\tau_i$ yields $V_1=V_2$ and we get the following equations.
\begin{equation}\label{S307}
-n_{23} + m_{11} n_{23}=0
\end{equation}
\begin{equation}\label{S308}
m_{12} n_{23}=0
\end{equation}
\begin{equation}\label{S309}
n_{32} - m_{11} n_{32}=0
\end{equation}
\begin{equation}\label{S310}
-m_{12} + m_{12} n_{33}=0
\end{equation}
\begin{equation}\label{S311}
m_{21} n_{32}=0
\end{equation}
\begin{equation}\label{S312}
m_{21} - m_{21} n_{33}=0
\end{equation}

\vspace*{0.1cm}

For the case (iv) $i=j+2$ we have, by direct computations, that
   $$\nu(\tau_i\sigma_j)=\left( \begin{array}{c|@{}c|c@{}}
   \begin{matrix}
     I_{i-1} 
   \end{matrix} 
      & 0 & 0 \\
      \hline
    0 &\hspace{0.2cm} \begin{matrix}
   		W_1
   		\end{matrix}  & 0  \\
\hline
0 & 0 & I_{n-i-3}
\end{array} \right) 
\text{ and } 
\nu(\sigma_j\tau_i)=\left( \begin{array}{c|@{}c|c@{}}
   \begin{matrix}
     I_{i-1} 
   \end{matrix} 
      & 0 & 0 \\
      \hline
    0 &\hspace{0.2cm} \begin{matrix}
   		W_2
   		\end{matrix}  & 0  \\
\hline
0 & 0 & I_{n-i-3}
\end{array} \right), 
$$ 
where 
$$W_1=\left(
\begin{array}{ccccc}
 m_{11} & m_{12} & 0 & 0 & 0 \\
 m_{21} & m_{22} & m_{23} & 0 & 0 \\
 0 & m_{32}n_{11} & m_{33}n_{11} & n_{12} &0 \\
 0 & m_{32}n_{21} & m_{33}n_{21} & n_{22} & n_{23} \\
 0 & 0 & 0 & n_{32} & n_{33} \\
\end{array}
\right)$$
and 
$$W_2=\left(
\begin{array}{ccccc}
 m_{11} & m_{12} & 0 & 0 & 0 \\
 m_{21} & m_{22} & m_{23}n_{11} & m_{23}n_{12} & 0 \\
  0 & m_{32} & m_{33}n_{11} & m_{33}n_{12} & 0 \\
 0 & 0 & n_{21} & n_{22} & n_{23}\\
 0 & 0 & 0 & n_{32} & n_{33} \\
\end{array}
\right).$$

\vspace*{0.1cm}

Hence, the relation $\tau_i\sigma_j=\sigma_j\tau_i$ yields $W_1=W_2$.  Consequently, we have the following equations.
\begin{equation}\label{S313}
m_{23} - m_{23} n_{11}=0
\end{equation}
\begin{equation}\label{S314}
m_{23} n_{12}=0
\end{equation}
\begin{equation}\label{S315}
-m_{32} + m_{32} n_{11}=0
\end{equation}
\begin{equation}\label{S316}
n_{12} - m_{33} n_{12}=0
\end{equation}
\begin{equation}\label{S317}
m_{32} n_{21}=0
\end{equation}
\begin{equation}\label{S318}
-n_{21} + m_{33} n_{21}=0
\end{equation}

\vspace*{0.1cm}

Now, direct computations show that for $1\leq i\leq n-1$ we have the following.

\vspace*{0.1cm}

$$\nu(\tau_i\sigma_i)=\left( \begin{array}{c|@{}c|c@{}}
   \begin{matrix}
     I_{i-1} 
   \end{matrix} 
      & 0 & 0 \\
      \hline
    0 &\hspace{0.2cm} \begin{matrix}
   		X_1
   		\end{matrix}  & 0  \\
\hline
0 & 0 & I_{n-i-3}
\end{array} \right) 
\text{ and }
\nu(\sigma_i\tau_i)=\left( \begin{array}{c|@{}c|c@{}}
   \begin{matrix}
     I_{i-1} 
   \end{matrix} 
      & 0 & 0 \\
      \hline
    0 &\hspace{0.2cm} \begin{matrix}
   		X_2
   		\end{matrix}  & 0  \\
\hline
0 & 0 & I_{n-i-3}
\end{array} \right),
$$ 
where
$$X_1=\left(
\begin{array}{ccc}
 m_{11} n_{11}+m_{21}n_{12} & m_{12} n_{11}+m_{22} n_{12} & m_{23} n_{12} \\
m_{11} n_{21}+m_{21} n_{22} & m_{12} n_{21}+m_{22} n_{22}+m_{32} n_{23} & m_{23} n_{22}+m_{33} n_{23} \\
m_{21} n_{32} & m_{22}n_{32}+m_{32}n_{33} & m_{23} n_{32}+m_{33} n_{33} \\
\end{array}
\right),$$
and  
$$X_2=\left(
\begin{array}{ccc}
 m_{11} n_{11}+m_{12}n_{21} & m_{11}n_{12}+m_{12} n_{22} & m_{12} n_{23} \\
 m_{21} n_{11}+m_{22} n_{21} & m_{21} n_{12}+m_{22}n_{22}+m_{23} n_{32} & m_{22} n_{23}+m_{23} n_{33} \\
 m_{32} n_{21} & m_{32} n_{22}+m_{33} n_{32} & m_{32} n_{23}+m_{33} n_{33} \\
\end{array}
\right).$$

\vspace*{0.1cm}

The relation $\tau_i\sigma_i=\sigma_i\tau_i$ for $1\leq i\leq n-1$ yields $X_1=X_2$ and we obtain the following equations.
\begin{equation}\label{S319}
-m_{21} n_{12} + m_{12} n_{21}=0
\end{equation}
\begin{equation}\label{S320}
-m_{12} n_{11} + m_{11} n_{12} - m_{22} n_{12} + m_{12} n_{22}=0
\end{equation}
\begin{equation}\label{S321}
-m_{23} n_{12} + m_{12} n_{23}=0
\end{equation}
\begin{equation}\label{S322}
m_{21} n_{11} - m_{11} n_{21} + m_{22} n_{21} - m_{21} n_{22}=0
\end{equation}
\begin{equation}\label{S323}
m_{21} n_{12} - m_{12} n_{21} - m_{32} n_{23} + m_{23} n_{32}=0
\end{equation}
\begin{equation}\label{S324}
-m_{23} n_{22} + m_{22} n_{23} - m_{33} n_{23} + m_{23} n_{33}=0
\end{equation}
\begin{equation}\label{S325}
m_{32} n_{21} - m_{21} n_{32}=0
\end{equation}
\begin{equation}\label{S326}
m_{32} n_{22} - m_{22} n_{32} + m_{33} n_{32} - m_{32} n_{33}=0
\end{equation}
\begin{equation}\label{S327}
m_{32} n_{23} - m_{23} n_{32}=0
\end{equation}

\vspace*{0.1cm}

By direct computations we get 
$$\nu'(\sigma_i\sigma_{i+1}\tau_{i})=\left( \begin{array}{c|@{}c|c@{}}
   \begin{matrix}
     I_{i-1} 
   \end{matrix} 
      & 0 & 0 \\
      \hline
    0 &\hspace{0.2cm} \begin{matrix}
   		Y_1
   		\end{matrix}  & 0  \\
\hline
0 & 0 & I_{n-i-3}
\end{array} \right)$$ 
and  
$$\nu'(\tau_{i+1}\sigma_i\sigma_{i+1})=\left( \begin{array}{c|@{}c|c@{}}
   \begin{matrix}
     I_{i-1} 
   \end{matrix} 
      & 0 & 0 \\
      \hline
    0 &\hspace{0.2cm} \begin{matrix}
   		Y_2
   		\end{matrix}  & 0  \\
\hline
0 & 0 & I_{n-i-3}
\end{array} \right),$$
where 
$Y_1=(y_{11},y_{12},y_{13},y_{14})$ and $Y_2=(y_{21},y_{22},y_{23},y_{24})$ with
$$y_{11}=\left(
\begin{array}{c}
 m_{11} m_{12} n_{21}+m_{11} n_{11} \\
 n_{21} (m_{11} m_{22}+m_{21} m_{23})+m_{21} n_{11} \\
 n_{21} (m_{11} m_{32}+m_{21} m_{33}) \\
 0   \\
\end{array}
\right),
$$

$$y_{12}=\left(
\begin{array}{c}
 m_{11} m_{12} n_{22}+m_{11} n_{12}+m_{12}^2 n_{32} \\
n_{22} (m_{11} m_{22}+m_{21} m_{23})+n_{32} (m_{12} m_{22}+m_{22} m_{23})+m_{21} n_{12} \\
n_{22} (m_{11} m_{32}+m_{21} m_{33})+n_{32} (m_{12} m_{32}+ m_{22} m_{33}) \\
m_{32} n_{32}  \\
\end{array}
\right),
$$

$$y_{13}=\left(
\begin{array}{c}
 m_{11} m_{12} n_{23}+m_{12}^2 n_{33} \\
 n_{23} (m_{11} m_{22}+m_{21} m_{23})+n_{33} (m_{12} m_{22}+ m_{22} m_{23}) \\
 n_{23} (m_{11} m_{32}+m_{21} m_{33})+n_{33} (m_{12} m_{32}+m_{22} m_{33}) \\
  m_{32} n_{33} \\
\end{array}
\right)
$$

$$y_{14}=\left(
\begin{array}{c}
 0 \\
 m_{23}^2 \\
 m_{23} m_{33} \\
 m_{33} \\
\end{array}
\right),
$$
 
$$y_{21}=\left(
\begin{array}{c}
 m_{11}  \\
 m_{21} n_{11} \\
 m_{21} n_{21} \\
 0 \\
\end{array}
\right),
$$

$$y_{22}=\left(
\begin{array}{c}
m_{11} m_{12} \\
m_{11} (m_{22} n_{11}+m_{32} n_{12})+m_{21} (m_{23} n_{11}+m_{33} n_{12}) \\
m_{11} (m_{22} n_{21}+m_{32} n_{22})+m_{21} (m_{23} n_{21}+m_{33} n_{22}) \\
m_{11} m_{32} n_{32}+m_{21} m_{33} n_{32} \\
\end{array}
\right),
$$

$$y_{23}=\left(
\begin{array}{c}
m_{12}^2 \\
m_{12} (m_{22} n_{11}+m_{32} n_{12})+m_{22} (m_{23} n_{11}+m_{33} n_{12}) \\
m_{12} (m_{22} n_{21}+m_{32} n_{22})+m_{22} (m_{23} n_{21}+m_{33} n_{22})+m_{32} n_{23} \\
m_{12} m_{32} n_{32}+m_{22} m_{33} n_{32}+m_{32} n_{33} \\
\end{array}
\right),$$

and 

$$y_{24}=\left(
\begin{array}{c}
 0 \\
m_{23} (m_{23} n_{11}+m_{33} n_{12}) \\
m_{23} (m_{23} n_{21}+m_{33} n_{22})+m_{33} n_{23} \\
m_{23} m_{33} n_{32}+m_{33} n_{33} \\
\end{array}
\right)$$
\vspace*{0.1cm}

The relation $\sigma_i\sigma_{i+1}\tau_i=\tau_{i+1}\sigma_{i}\sigma_{i+1}$ implies that $Y_1=Y_2$ and we get the following equations.
\begin{equation}\label{S328}
m_{11} (-1 + n_{11} + m_{12} n_{21})=0
\end{equation}
\begin{equation}\label{S329}
(n_{12} + m_{12} (-1 + n_{22})) + m_{12}^2 n_{32}=0
\end{equation}
\begin{equation}\label{S330}
m_{12} (m_{11} n_{23} + m_{12} (-1 + n_{33}))=0
\end{equation}
\begin{equation}\label{S331}
(m_{11} m_{22} + m_{21} m_{23}) n_{21}=0
\end{equation}
\begin{equation}\label{S332}
\begin{aligned}
-m_{11} (m_{32} n_{12} + m_{22} (n_{11} - n_{22})) +
 m_{21} (-m_{23} n_{11} + n_{12} - m_{33} n_{12} + m_{23} n_{22})\\
 + m_{22} (m_{12} + m_{23}) n_{32}=0
 \end{aligned}
\end{equation}
\begin{equation}\label{S333}
\begin{aligned}
m_{21} m_{23} n_{23} - m_{12} (m_{32} n_{12} + m_{22} (n_{11} - n_{33})) \\
+m_{22} (-m_{23} n_{11} - m_{33} n_{12} + m_{11} n_{23} + m_{23}n_{33})=0
\end{aligned}
\end{equation}
\begin{equation}\label{S334}
m_{23} (m_{23} - m_{23} n_{11} - m_{33} n_{12})=0
\end{equation}
\begin{equation}\label{S335}
(m_{11} m_{32} + m_{21} (-1 + m_{33})) n_{21}=0
\end{equation}
\begin{equation}\label{S336}
-m_{11} m_{22} n_{21} - m_{21} m_{23} n_{21} + m_{12} m_{32} n_{32} + m_{22} m_{33} n_{32}=0
\end{equation}
\begin{equation}\label{S337}
\begin{aligned}
(-m_{32} + m_{11} m_{32} + m_{21} m_{33}) n_{23} - m_{12} (m_{22} n_{21} + m_{32} (n_{22} - n_{33}))\\
 -m_{22} (m_{23} n_{21} + m_{33} n_{22} - m_{33} n_{33})=0
\end{aligned}
\end{equation}
\begin{equation}\label{S338}
m_{23} m_{33} - m_{23} (m_{23} n_{21} + m_{33} n_{22}) - m_{33} n_{23}=0
\end{equation}
\begin{equation}\label{S339}
(m_{32} - m_{11} m_{32} - m_{21} m_{33}) n_{32}=0
\end{equation}
\begin{equation}\label{S340}
(m_{12} m_{32} + m_{22} m_{33}) n_{32}=0
\end{equation}
\begin{equation}\label{S341}
m_{33} (-1 + m_{23} n_{32} + n_{33})=0
\end{equation}

\vspace*{0.1cm}

By direct computations we get 
$$\nu'(\sigma_{i+1}\sigma_i\tau_{i+1})=\left( \begin{array}{c|@{}c|c@{}}
   \begin{matrix}
     I_{i-1} 
   \end{matrix} 
      & 0 & 0 \\
      \hline
    0 &\hspace{0.2cm} \begin{matrix}
   		Z_1
   		\end{matrix}  & 0  \\
\hline
0 & 0 & I_{n-i-3}
\end{array} \right)$$ 
and  
$$\nu'(\tau_i\sigma_{i+1}\sigma_i)=\left( \begin{array}{c|@{}c|c@{}}
   \begin{matrix}
     I_{i-1} 
   \end{matrix} 
      & 0 & 0 \\
      \hline
    0 &\hspace{0.2cm} \begin{matrix}
   		Z_2
   		\end{matrix}  & 0  \\
\hline
0 & 0 & I_{n-i-3}
\end{array} \right),$$
where 
$Z_1=(z_{11},z_{12},z_{13},z_{14})$ and $Z_2=(z_{21},z_{22},z_{23},z_{24})$. Here, 
$$z_{11}=
\left(
\begin{array}{c}
 m_{11} \\
 m_{11} m_{21} \\
 m_{21}^2 \\
 0  \\
\end{array}
\right),
$$

$$z_{12}=
\left(
\begin{array}{c}
 m_{12} n_{11} \\
n_{11} (m_{11} m_{22}+m_{12} m_{32})+n_{21} (m_{11} m_{23}+m_{12} m_{33}) \\
n_{21} (m_{21} m_{23}+m_{22} m_{33})+n_{11} (m_{21} m_{22}+m_{22} m_{32}) \\
m_{32}^2 n_{11}+m_{32} m_{33} n_{21} \\
\end{array}
\right),
$$

$$z_{13}=\left(
\begin{array}{c}
 m_{12} n_{12}\\
 n_{12} (m_{11} m_{22}+m_{12} m_{32})+n_{22} (m_{11} m_{23}+m_{12} m_{33}) \\
 n_{22} (m_{21} m_{23}+m_{22} m_{33})+n_{12} (m_{21} m_{22}+m_{22} m_{32})+m_{23} n_{32} \\
 m_{32}^2 n_{12}+m_{32} m_{33} n_{22}+m_{33} n_{32} \\
\end{array}
\right),
$$

$$z_{14}=\left(
\begin{array}{c}
  0 \\
 n_{23} (m_{11} m_{23}+m_{12} m_{33}) \\
 n_{23} (m_{21} m_{23}+m_{22} m_{33})+m_{23} n_{33} \\
 m_{32} m_{33} n_{23}+m_{33} n_{33} \\
\end{array}
\right),
$$
 
$$z_{21}=\left(
\begin{array}{c}
 m_{11} m_{21} n_{12}+m_{11} n_{11} \\
 m_{21} (m_{11} n_{22}+m_{21} n_{23})+m_{11} n_{21} \\
 m_{21} (m_{11} n_{32}+m_{21} n_{33}) \\
 0  \\
\end{array}
\right),
$$

$$z_{22}=\left(
\begin{array}{c}
m_{11} m_{22} n_{12}+m_{12} m_{32} n_{12}+m_{12} n_{11} \\
m_{22} (m_{11} n_{22}+m_{21} n_{23})+m_{32} (m_{12} n_{22}+m_{22} n_{23})+m_{12} n_{21} \\
m_{22} (m_{11} n_{32}+m_{21} n_{33})+m_{32} (m_{12} n_{32}+m_{22} n_{33}) \\
m_{32}^2 \\
\end{array}
\right),
$$

$$z_{23}=\left(
\begin{array}{c}
 m_{11} m_{23} n_{12}+m_{12} m_{33} n_{12} \\
 m_{23} (m_{11} n_{22}+m_{21} n_{23})+m_{33} (m_{12} n_{22}+m_{22} n_{23}) \\
 m_{23} (m_{11} n_{32}+m_{21} n_{33})+m_{33} (m_{12} n_{32}+m_{22} n_{33}) \\
 m_{32} m_{33} \\
\end{array}
\right),
$$
 and 
 
 $$z_{24}=\left(
\begin{array}{c}
 0 \\
 m_{23} n_{23} \\
 m_{23} n_{33} \\
 m_{33} \\
\end{array}
\right).
$$
\vspace*{0.1cm}

The relation $\sigma_{i+1}\sigma_i\tau_{i+1}=\tau_i\sigma_{i+1}\sigma_i$ implies that $Z_1=Z_2$ and we get the following equations.
\begin{equation}\label{S342}
m_{11} - m_{11} n_{11} - m_{11} m_{21} n_{12}=0
\end{equation}
\begin{equation}\label{S343}
-m_{11} m_{22} n_{12} - m_{12} m_{32} n_{12}=0
\end{equation}
\begin{equation}\label{S344}
m_{12} n_{12} - m_{11} m_{23} n_{12} - m_{12} m_{33} n_{12}=0
\end{equation}
\begin{equation}\label{S345}
m_{11} m_{21} - m_{11} n_{21} - m_{21} (m_{11} n_{22} + m_{21} n_{23})=0
\end{equation}
\begin{equation}\label{S346}
\begin{aligned}
&(m_{11} m_{22} + m_{12} m_{32}) n_{11} - m_{12} n_{21} + (m_{11} m_{23} + m_{12} m_{33}) n_{21} \\
& -m_{22} (m_{11} n_{22} + m_{21} n_{23}) - m_{32}(m_{12} n_{22} + m_{22} n_{23})=0.
\end{aligned}
\end{equation}
\begin{equation}\label{S347}
\begin{aligned}
& (m_{11} m_{22} + m_{12} m_{32}) n_{12} + (m_{11} m_{23} + m_{12} m_{33}) n_{22}\\
& -m_{23} (m_{11} n_{22} + m_{21} n_{23}) - m_{33} (m_{12} n_{22} + m_{22} n_{23})=0.
\end{aligned}
\end{equation}
\begin{equation}\label{S348}
-m_{23} n_{23} + (m_{11} m_{23} + m_{12} m_{33}) n_{23}=0
\end{equation}
\begin{equation}\label{S349}
m_{21}^2 - m_{21} (m_{11} n_{32} + m_{21} n_{33})=0
\end{equation}
\begin{equation}\label{S350}
\begin{aligned}
& (m_{21} m_{22} + m_{22} m_{32}) n_{11} + (m_{21} m_{23} + m_{22} m_{33}) n_{21} \\
& -m_{22} (m_{11} n_{32} + m_{21} n_{33}) - m_{32} (m_{12} n_{32} + m_{22} n_{33})=0.
\end{aligned}
\end{equation}
\begin{equation}\label{S351}
\begin{aligned}
&(m_{21} m_{22} + m_{22} m_{32}) n_{12} + (m_{21} m_{23} + m_{22} m_{33}) n_{22} + m_{23} n_{32} \\
& -m_{23} (m_{11} n_{32} + m_{21} n_{33}) - m_{33} (m_{12} n_{32} + m_{22} n_{33})=0.
\end{aligned}
\end{equation}
\begin{equation}\label{S352}
(m_{21} m_{23} + m_{22} m_{33}) n_{23}=0
\end{equation}
\begin{equation}\label{S353}
-m_{32}^2 + m_{32}^2 n_{11} + m_{32} m_{33} n_{21}=0
\end{equation}
\begin{equation}\label{S354}
-m_{32} m_{33} + m_{32}^2 n_{12} + m_{32} m_{33} n_{22} + m_{33} n_{32}=0
\end{equation}
\begin{equation}\label{S355}
-m_{33} + m_{32} m_{33} n_{23} + m_{33} n_{33}=0
\end{equation}

\vspace*{0.1cm}

Theorem \ref{repBn} classifies the non-trivial homogeneous $3$-local representations of $B_n$ ($n>3$) into eight cases. For each case, we solve, using Mathematica software, the system $\Sigma$ formed by the equations numbered (\ref{S301}) through (\ref{S355}), using the given matrix $M=M_j$ to derive the corresponding matrix $N=N_j$, $j=1,2,\ldots,8$.\\

\textbf{Case 1.} $m_{11}=1, m_{12}=0, m_{32}=\frac{1-m{_22}}{m_{23}}, m_{33}=0$, $m_{22}\neq1$ and $m_{23}\neq0$. Then $M_1=\left(
\begin{array}{ccc}
 1 & 0 & 0 \\
 0 & m_{22} & m_{23} \\
 0 & \frac{1-m_{22}}{m_{23}} & 0 \\
\end{array}
\right)$. Hence, $\Sigma$ has two solutions: 
\begin{itemize}
\item $m_{22} = 0, n_{11} = 1; n_{12} = 0, n_{21} = 0, n_{23} = m_{23}^2 n_{32}, n_{33} = n_{22},$
\item $ n_{11} = 1, n_{12} = 0, n_{21} = 0, n_{22} = 1 - m_{23} n_{32}, n_{23} = -\frac{m_{23}^2 n_{32}}{-1 + m_{22}}, n_{33} = \frac{-1 + m_{22} + m_{23} n_{32}}{-1 + m_{22}}.$
\end{itemize}
If $m_{22}=0$, the matrices $M_1$ and $N_1$ become.
$$M_1=\left(\begin{array}{ccc}
 1 & 0 & 0 \\
 0 & 0 & m_{23} \\
 0 & \frac{1}{m_{23}} & 0 \\
\end{array}
\right)
 \text{ and }
 N_1=\left(
\begin{array}{ccc}
 1 & 0 & 0 \\
 0 & n_{22} & m_{23}^2 n_{32} \\
 0 & n_{32} & n_{22} \\
\end{array}
\right).
 $$
If $m_{22}\neq0$, the matrices $M$ and $N$ become
 $$M_1=\left(
\begin{array}{ccc}
 1 & 0 & 0 \\
 0 & m_{22} & m_{23} \\
 0 & \frac{1-m_{22}}{m_{23}} & 0 \\
\end{array}
\right)
 \text{ and }
 N_1=\left(
\begin{array}{ccc}
 1 & 0 & 0 \\
 0 & 1-m_{23} n_{32} & -\frac{m_{23}^2 n_{32}}{m_{22}-1} \\
 0 & n_{32} & \frac{m_{22}+m_{23} n_{32}-1}{m_{22}-1} \\
\end{array}
\right).$$ Thus, the result is as required.
\\

\textbf{Case 2.} $m_{11} = 0, m_{21} = \frac{1 - m_{22}}{m_{12}}, m_{23} = 0, m_{32} = 0, m_{33} = 1$, $m_{22}\neq1$ and $m_{12}\neq0$.  Then $M_2=\left(
\begin{array}{ccc}
 0 & m_{12} & 0 \\
 \frac{1-m_{22}}{m_{12}} & m_{22} & 0 \\
 0 & 0 & 1 \\
\end{array}
\right)$. Hence, $\Sigma$ has two solutions:
\begin{itemize}
\item $m_{22} = 0, n_{22} = n_{11}, n_{12} = m_{12}^2 n_{21}, n_{23} = 0, n_{32} = 0, n_{33} = 1,$
\item $n_{11} = \frac{-1 + m_{22} + m_{12} n_{21}}{-1 + m_{22}}, n_{12} = -\frac{m_{12}^2 n_{21}}{-1 + m_{22}}, n_{22} = 1 - m_{12} n_{21}, n_{23} = 0, n_{32} = 0, n_{33} = 1$.
\end{itemize}
If $m_{22}=0$, the matrices $M_2$ and $N_2$ become
$$M_2=\left(
\begin{array}{ccc}
 0 & m_{12} & 0 \\
 \frac{1}{m_{12}} & 0 & 0 \\
 0 & 0 & 1 \\
\end{array}
\right)
\text{ and }
N_2=\left(
\begin{array}{ccc}
 n_{11} & m_{12}^2 n_{21} & 0 \\
 n_{21} & n_{11} & 0 \\
 0 & 0 & 1 \\
\end{array}
\right). $$
If $m_{22}\neq0$, the matrices $M_2$ and $N_2$ become
$$M_2=\left(
\begin{array}{ccc}
 0 & m_{12} & 0 \\
 \frac{1-m_{22}}{m_{12}} & m_{22} & 0 \\
 0 & 0 & 1 \\
\end{array}
\right)
\text{ and }
N_2=\left(
\begin{array}{ccc}
 \frac{m_{12} n_{21}+m_{22}-1}{m_{22}-1} & -\frac{m_{12}^2 n_{21}}{m_{22}-1} & 0 \\
 n_{21} & 1-m_{12} n_{21} & 0 \\
 0 & 0 & 1 \\
\end{array}
\right).$$
Thus, the result is as required.\\

\textbf{Case 3.} $m_{11} = 1, m_{12} = -\frac{m_{22}}{m_{32}}, m_{21} = 0, m_{23} = 0, m_{33} = 1,$ and $m_{22}m_{32}\neq0$. Then $M_3=\left(
\begin{array}{ccc}
 1 & -\frac{m_{22}}{m_{32}} & 0 \\
 0 & m_{22} & 0 \\
 0 & m_{32} & 1 \\
\end{array}
\right)$ and $\Sigma$ has one solution

$$n_{11}=1, n_{21}=0,n_{22}=\frac{m_{22} (-m_{32}) n_{12}+m_{22}+m_{32} n_{12}}{m_{22}},n_{23}=0,n_{32}=-\frac{m_{32}^2n_{12}}{m_{22}},n_{33}=1.$$
Therefore,
$N_3=\left(
\begin{array}{ccc}
 1 & n_{12} & 0 \\
 0 & \frac{m_{32} n_{12}}{m_{22}}-m_{32} n_{12}+1 & 0 \\
 0 & -\frac{m_{32}^2 n_{12}}{m_{22}} & 1 \\
\end{array}
\right),$ as required.\\
 
\textbf{Case 4.} $m_{11}=1, m_{12}=0, m_{21}=-\frac{m_{22}}{m_{23}}, m_{32}=0, m_{33}=1,$ and $ m_{23}m_{22}\neq0.$ Then $M_4=\left(
\begin{array}{ccc}
 1 & 0 & 0 \\
 -\frac{m_{22}}{m_{23}} & m_{22} & m_{23} \\
 0 & 0 & 1 \\
\end{array}
\right)$ and $\Sigma$ has one solution

$$n_{11}=1, n_{12}=0, n_{22}=\frac{m_{22} (-m_{23}) n_{21}+m_{22}+m_{23} n_{21}}{m_{22}}, n_{23}=-\frac{m_{23}^2 n_{21}}{m_{22}}, n_{32}=0, n_{33}=1.$$
Therefore 
$N_4=\left(
\begin{array}{ccc}
 1 & 0 & 0 \\
 n_{21} & \frac{m_{23} n_{21}}{m_{22}}-m_{23} n_{21}+1 & -\frac{m_{23}^2 n_{21}}{m_{22}} \\
 0 & 0 & 1 \\
\end{array}
\right),$ as required.\\

\textbf{Case 5.} $m_{11}=1, m_{12}=0, m_{21}=0, m_{22}=0, m_{33}=1-m_{23} m_{32},$ and $m_{23}m_{32}\neq0$. Then $M_5=\left(
\begin{array}{ccc}
 1 & 0 & 0 \\
 0 & 0 & m_{23} \\
 0 & m_{32} & 1-m_{23} m_{32} \\
\end{array}
\right)$. Hence, $\Sigma$ has the two solutions
\begin{itemize}
\item $m_{32}=\frac{1}{m_{23}}, n_{11}=1, n_{12}=0, n_{21}=0, n_{32}=\frac{n_{23}}{m_{23}^2}, n_{33}=n_{22},$
\item $n_{11}=1, n_{12}=0, n_{21}=0, n_{23}=m_{23}-m_{23} n_{22},n_{32}=m_{32}-m_{32} n_{22}, n_{33}=m_{23} m_{32} n_{22}-m_{23} m_{32}+1.$
\end{itemize}
If $m_{32}m_{23}=1$, the matrices $M_5$ and $N_5$ become
$$M_5=
\left(
\begin{array}{ccc}
 1 & 0 & 0 \\
 0 & 0 & m_{23} \\
 0 & \frac{1}{m_{23}} & 0 \\
\end{array}
\right)
\text{ and }
N_5=\left(
\begin{array}{ccc}
 1 & 0 & 0 \\
 0 & n_{22} & n_{23} \\
 0 & \frac{n_{23}}{m_{23}^2} & n_{22} \\
\end{array}
\right).$$
If $m_{32}m_{23}\neq1$, the matrices $M_5$ and $N_5$ become
$$M_5=\left(
\begin{array}{ccc}
 1 & 0 & 0 \\
 0 & 0 & m_{23} \\
 0 & m_{32} & 1-m_{23} m_{32} \\
\end{array}
\right)
\text{ and }
N_5=\left(
\begin{array}{ccc}
 1 & 0 & 0 \\
 0 & n_{22} & m_{23}-m_{23} n_{22} \\
 0 & m_{32}-m_{32} n_{22} & m_{23} m_{32} (n_{22}-1)+1 \\
\end{array}
\right).$$
Thus, the result is as required.\\

\textbf{Case 6.}
$m_{11}=1-m_{12} m_{21}, m_{22}=0, m_{23}=0, m_{32}=0, m_{33}=1,$ and $m_{12}m_{21}\neq0.$ Then $M_6=\left(
\begin{array}{ccc}
 1-m_{12} m_{21} & m_{12} & 0 \\
 m_{21} & 0 & 0 \\
 0 & 0 & 1 \\
\end{array}
\right)$. Hence, $\Sigma$ has two solutions.
\begin{itemize}
\item $m_{21}=\frac{1}{m_{12}}, n_{11}=n_{22}, n_{21}=\frac{n_{12}}{m_{12}^2}, n_{23}=0, n_{32}=0, n_{33}=1$,
\item $n_{11}=m_{12} m_{21} n_{22}-m_{12} m_{21}+1,n_{12}=m_{12}-m_{12} n_{22},n_{21}=m_{21}-m_{21} n_{22}, n_{23}=0, n_{32}=0, n_{33}=1.$
\end{itemize}
If $m_{12}m_{21}=1$, the matrices $M_6$ and $N_6$ become
$$M_6=\left(
\begin{array}{ccc}
 0 & m_{12} & 0 \\
 \frac{1}{m_{12}} & 0 & 0 \\
 0 & 0 & 1 \\
\end{array}
\right) 
\text{ and }
N_6=\left(
\begin{array}{ccc}
 n_{22} & n_{12} & 0 \\
 \frac{n_{12}}{m_{12}^2} & n_{22} & 0 \\
 0 & 0 & 1 \\
\end{array}
\right).$$
If $m_{12}m_{21}\neq1$, the matrices $M_6$ and $N_6$ become
$$M_6=\left(
\begin{array}{ccc}
 1-m_{12} m_{21} & m_{12} & 0 \\
 m_{21} & 0 & 0 \\
 0 & 0 & 1 \\
\end{array}
\right) 
\text{ and }
N_6=\left(
\begin{array}{ccc}
 m_{12} m_{21} (n_{22}-1)+1 & m_{12}-m_{12} n_{22} & 0 \\
 m_{21}-m_{21} n_{22} & n_{22} & 0 \\
 0 & 0 & 1 \\
\end{array}
\right).$$
Thus, the result is as required.\\

\textbf{Case 7.}
$m_{11} = 1, m_{12} = 0, m_{21} = 0, m_{22} = 0, m_{33} = 0$ and $m_{23}m_{32}\neq0$. Then $M_7=\left(
\begin{array}{ccc}
 1 & 0 & 0 \\
 0 & 0 & m_{23} \\
 0 & m_{32} & 0 \\
\end{array}
\right)$ and $\Sigma$ has the following solution

$n_{11}=1, n_{12}=n_{21}=0, n_{32}=\frac{m_{32}n_{23}}{m_{23}}, n_{33}=n_{22}$. Thus, $N_7=\left(
\begin{array}{ccc}
 1 & 0 & 0 \\
 0 & n_{22} & n_{23} \\
 0 & \frac{m_{32}n_{23}}{m_{23}} & n_{22} \\
\end{array}
\right),$ as required.

\textbf{Case 8.}
$m_{11} = 0, m_{22} = 0, m_{23} = 0, m_{32} = 0, m_{33} = 1$, and $m_{12}m_{21}\neq0$. Then $M_8=\left(
\begin{array}{ccc}
 1 & m_{12} & 0 \\
 m_{21} & 0 & 0 \\
 0 & 0 & 0 \\
\end{array}
\right)$ and $\Sigma$ has the following solution

$n_{21}=\frac{m_{21}n_{12}}{m_{12}}, n_{22}=n_{11}, n_{23}=0, n_{32}=0, n_{33}=1$. Thus, 
$N_8=\left(
\begin{array}{ccc}
 n_{11} & n_{12} & 0 \\
 \dfrac{m_{21} n_{12}}{m_{12}} & n_{11} & 0 \\
 0 & 0 & 1 \\
\end{array}
\right),$ as required.
\end{proof}

\vspace*{0.1cm}

In Theorems \ref{repBn} and \ref{repSMn}, we classify all non-trivial homogeneous $3$-local representations of $B_n$ and their homogeneous $3$-local extensions to $SM_n$ for all $n\geq 4$. The remaining case is for $n=3$. This leads to ask the following question.

\vspace*{0.1cm}

\begin{question}
Let $\nu :B_3\to GL_{4}(\mathbb{C})$ be a non-trivial homogeneous $3$-local representation of $B_3$. Let $\nu' :SM_3\to M_{4}(\mathbb{C})$ be a non-trivial homogeneous $3$-local representation of $SM_3$ which is an extension of $\nu$. What are all possible classifications of $\nu$ and $\nu'?$
\end{question}

\vspace*{0.1cm}

Now, we notice that the $F$-representation is shown to be reducible in Theorem \ref{Fre}, but we do not know whether its homogeneous $3$-local extension to $SM_n$ is reducible or not. This leads us to ask the following question.

\vspace*{0.1cm}

\begin{question}
Is it true that the homogeneous $3$-local extension of the $F$-representation of $B_n$ to $SM_n$ is reducible?
\end{question}

\vspace*{0.1cm}

Also, in a similar way as in Question \ref{qq} in the previous section, we can ask the following question.

\vspace*{0.1cm}

\begin{question}
Can we find a homogeneous $3$-local representation of $B_n$ which is reducible and its homogeneous $3$-local extension to $SM_n$ is irreducible?
\end{question}

\vspace*{0.1cm}

\section{Conclusions} 
In this paper, we worked on classifying some local representations of the braid group $B_n$ and the singular braid monoid $SM_n$, which is a generalization of the work of Mikhalchishina in two ways. First, as Mikhalchishina classified all homogeneous $2$-local representations of $B_n$ for all $n\geq 3$, we classified all homogeneous $3$-local representations of $B_n$ for all $n \geq 4$. Second, we extend the result of Mikhalchishina and our first result to $SM_n$; that is, we classified all homogeneous $2$-local representations of $SM_n$ for all $n\geq 2$ and all homogeneous $3$-local representations of $SM_n$ for all $n\geq 4$. The third result is that we classified all $\Phi$-type extensions of all homogeneous $2$-local representation of $B_n$ to $SM_n$ for all $n\geq 4$. Lastly, we offered a few significant questions for further research.
\vspace*{0.1cm}


\end{document}